 \documentclass[letterpaper,11pt]{article}

\usepackage[usenames,dvipsnames,svgnames,table]{xcolor}
\usepackage[colorlinks=true,linkcolor=blue,citecolor=ForestGreen]{hyperref}
\usepackage{amsmath,amssymb,amsfonts,amsthm}
\usepackage{enumerate}
\usepackage{scrextend}
\usepackage[utf8]{inputenc}
\usepackage[T1]{fontenc}
\usepackage[capitalize]{cleveref}
\usepackage{bm}
\usepackage{orcidlink}
\usepackage{fullpage}	
	\allowdisplaybreaks
\usepackage{tikz}

\Crefname{figure}{Figure}{Figures}
\Crefname{claim}{Claim}{Claims}
\Crefname{theorem}{Theorem}{Theorems}

\crefformat{equation}{\textup{#2(#1)#3}}
\crefrangeformat{equation}{\textup{#3(#1)#4--#5(#2)#6}}
\crefmultiformat{equation}{\textup{#2(#1)#3}}{ and \textup{#2(#1)#3}}
{, \textup{#2(#1)#3}}{, and \textup{#2(#1)#3}}
\crefrangemultiformat{equation}{\textup{#3(#1)#4--#5(#2)#6}}%
{ and \textup{#3(#1)#4--#5(#2)#6}}{, \textup{#3(#1)#4--#5(#2)#6}}{, and \textup{#3(#1)#4--#5(#2)#6}}

\Crefformat{equation}{Equation~\textup{#2(#1)#3}}
\Crefrangeformat{equation}{Equations~\textup{#3(#1)#4--#5(#2)#6}}
\Crefmultiformat{equation}{Equations~\textup{#2(#1)#3}}{ and \textup{#2(#1)#3}}
{, \textup{#2(#1)#3}}{, and \textup{#2(#1)#3}}
\Crefrangemultiformat{equation}{Equations~\textup{#3(#1)#4--#5(#2)#6}}%
{ and \textup{#3(#1)#4--#5(#2)#6}}{, \textup{#3(#1)#4--#5(#2)#6}}{, and \textup{#3(#1)#4--#5(#2)#6}}

\crefname{section}{Section}{Sections}
\crefname{theorem}{Theorem}{Theorems}
\crefname{claim}{Claim}{Claims}
\crefname{lemma}{Lemma}{Lemmas}

\usepackage{tikz}
\tikzset{normalnode/.style={circle, draw, fill=black, inner sep=0, minimum width=1.5mm}}

\usetikzlibrary{fit}
\usetikzlibrary{shapes,arrows,shadows,positioning}
\usetikzlibrary{backgrounds}

\newtheorem{theorem}{Theorem}[section]

\newtheorem{corollary}[theorem]{Corollary}

\newtheorem{lemma}[theorem]{Lemma}
\newtheorem{fact}[theorem]{Fact}

\theoremstyle{definition}
\newtheorem{definition}[theorem]{Definition}

\theoremstyle{remark}
\newtheorem{remark}[theorem]{Remark}

\renewcommand{\le}{\leqslant}
\renewcommand{\ge}{\geqslant}
\renewcommand{\leq}{\leqslant}
\renewcommand{\geq}{\geqslant}

\DeclareMathOperator{\Mallow}{Mallows}
\DeclareMathOperator{\Mallows}{Mallows_{seq}}

\newcommand{\bE}{\mathbb{E}}
\newcommand{\bN}{\mathbb{N}}
\newcommand{\bS}{\mathbf{S}}

\newcommand{\comment}[1]{}

\newcommand{\product}{\Lambda}
\newcommand{\productO}{\Lambda(\cS)}
\newcommand{\sumI}{\mathbf{I}}
\newcommand{\Arr}{\ensuremath{\operatorname{\mathbf{Arr}}}}
\newcommand{\Pra}[1]{\mathbb{P}\left(#1\right)}
\newcommand{\bB}{\mathbf{B}}
\newcommand{\cA}{\mathcal{A}}
\newcommand{\cB}{\mathcal{B}}
\newcommand{\Ind}{\mathbb{I}}

\newcommand{\barpi}{\bar \varPi}
\newcommand{\barPi}{\bar {\mathbf \Pi}}
\newcommand{\barbarPi}{{\bar{\bar {\mathbf \Pi}}}}
\newcommand{\NN}{\mathbb{N}}
\newcommand{\cS}{\mathcal{S}}
\newcommand{\cE}{\mathcal{E}}
\newcommand{\dtv}{\operatorname{d}_{TV}}
\newcommand{\bW}{\mathbf{W}}
\newcommand{\subseq}{\prec}
\begin{document}

\title{Arrangements of Consecutive Numbers in Mallows Permutations}

\author{Katarzyna Rybarczyk\thanks{Faculty of Mathematics and Computer Science, Adam Mickiewicz University, 60-614 Pozna\'n, Poland, \texttt{katarzyna.rybarczyk@amu.edu.pl}, 
} 
}

\date{}

\maketitle
\begin{abstract} We study the random variable that counts the number of specific arrangements of clustered consecutive numbers in permutations under the Mallows distribution. We provide an asymptotic expression for the expected value of this random variable. This result extends and tightens the previously known result by Pinsky (2022) concerning clustered consecutive numbers in Mallows permutations. Moreover, we identify a range of parameters for which the distribution of the number of arrangements of clustered consecutive numbers in Mallows permutations is close to a Poisson distribution.
\end{abstract} 	

\section{Introduction}

Let $\bS_n$ be the family of all permutations of $[n]=\{1,\ldots,n\}$. Given $n\in \bN$ and $q\in (0,\infty]$, the Mallows distribution $\Mallow(n,q)$ samples a random permutation $\varPi_n$ from $\bS_n$ according to the probability measure
\[\forall_{\sigma\in \bS_n}\Pra{\varPi_n=\sigma}=\frac{q^{\text{inv}(\sigma)}}{a_{n,q}},\]
where 
\[
\text{inv}(\sigma)=|\{(i,j):1\le i<j \le n, \sigma(i)>\sigma(j)\}|,
\]
and $a_{n,q}=\sum_{\pi\in \bS_n}q^{\text{inv}(\pi)}$ is a normalising constant. 

Since its introduction by Mallows~\cite{Mallows1957}, the model has attracted much attention. It has been considered in various contexts such as, for example, Markov chains \cite{BenjaminiEtal2005,DiaconisEtal2000}, finitely dependent colourings of the integers \cite{HolroydEtal2020}, stable matchings \cite{AngelEtal2021}, random binary search trees \cite{AddarioBerryEtal2021}, $q$--analogs of exchangeability \cite{GnedinOlshanski2010,GnedinOlshanski2012}, determinantal point processes \cite{BorodinETAL2010}, statistical physics \cite{Starr2009,StarrWalters2018}, genomics \cite{FangETAL2021}, random graphs \cite{Spagetti}, logical limit laws \cite{MullerSkermanVer2025}, permuton limits \cite{JasinskaRath2025}, and more. In the context of this article, the most closely related results are those concerning patterns \cite{CraneDesalvo2017,CrameDeSalvoElizalde2018,Pinsky2021,Dubach2026}, the number of descents \cite{He2021}, the longest monotone subsequence \cite{BasuBhat17,BhatnagarPeled2015,MuellerStarr2013}, the longest common subsequence of two Mallows permutations \cite{Jin2019}, and the cycle structure \cite{GladkichPeled2018,HeMV23}.

We investigate the clustering of consecutive elements under the Mallows distribution. This problem, for Mallows and $p$--shifted distributions, has been considered by Pinsky in \cite{Pinsky2022}. For a permutation $\varPi_n \sim \Mallow(n,q)$, Pinsky analyzes the random variable counting the number of intervals $\{j,j+1,\ldots,j+\ell-1\} \subseteq [n]$ whose images under $\varPi_n$ form consecutive natural numbers. Moreover, it is noted in \cite{Pinsky2022} that, with some additional work, these results can be extended to more general arrangements of consecutive elements. Our objective is to formulate our results at the highest possible level of generality. To this end, we introduce a precise definition of arrangements of clustered consecutive numbers in a permutation.

\begin{definition}\label{Def:Sarr}
	Let $\ell\ge 2$  be a constant integer and $\cS$, $\emptyset\neq\cS\subseteq \bS_\ell$, be a family of permutations of $[\ell]=\{1,\ldots,\ell\}$. 
	Moreover let $\varPi_n\in \bS_n$, $n\ge \ell$, $0\le k_1\le n-\ell$, and $0\le k_2\le n-\ell$. There is  an $(\cS,k_1,k_2)$--arrangement in  $\varPi_n$ if
	\[
	\exists_{\sigma\in \cS}\forall_{i\in [\ell]}\varPi_n(k_2+i)=k_1+\sigma(i).\]
	For any , $0\le k_1\le n-\ell$, $0\le k_2\le n-l$, an $(\cS,k_1,k_2)$--arrangement we call an $\cS$--arrangement.    	
\end{definition}
In other words, in an $(\cS,k_1,k_2)$--arrangement in $\varPi_n$, the consecutive numbers $k_2+1, \ldots, k_2+\ell$ take values in the set $\{k_1+1,\ldots,k_1+\ell\}$ in the order determined by a permutation $\sigma$ from $\cS$.
When we set $\cS=\bS_\ell$, the presence of an $(\cS,k_1,k_2)$--arrangement in $\varPi_n$ implies that the values $\varPi_n(k_2+1), \varPi_n(k_2+2),\ldots,\varPi_n(k_2+\ell)$ are $\ell$ consecutive numbers $k_1+1,\ldots,k_1+\ell$ (in any possible order), which is exactly the case considered in~\cite{Pinsky2022}.
Given $\emptyset\neq\cS\subseteq \bS_\ell$, denote by $\Arr_{\cS}(\varPi_n)$ the number of $\cS$--arrangements in $\varPi_n$. Then the random variable studied in \cite{Pinsky2022} is $\Arr_{\bS_\ell}(\varPi_n)$.

The Mallows distribution satisfies a duality relation. Let the reverse of a permutation $\sigma=(\sigma_1,\ldots,\sigma_n)$ be defined by
\[\sigma^{rev}=(\sigma_n,\ldots,\sigma_1).\]
For $\varPi_n\sim \Mallow(n,q)$ and $\varPi_n'\sim \Mallow(n,1/q)$, we have
\[\Pra{\varPi_n=\sigma}=\Pra{\varPi_n'=\sigma^{rev}}.\]

Therefore, all theorems concerning $\cS$--arrangements for $\varPi_n \sim \Mallow(n,q)$ with $q \in (0,1)$ can be reformulated in terms of $\cS^{rev}$--arrangements, where $\cS^{rev} = \{\sigma^{rev} \in \bS_{\ell} : \sigma \in \cS\}$, and $\varPi_n \sim \Mallow(n,1/q)$. Consequently, without loss of generality, the analysis may be restricted to the regime $q \in (0,1)$.

Moreover, for $q\in (0,1)$, the Mallows distribution assigns higher probability to permutations with fewer inversions, whereas for $q>1$ it favours permutations with a larger number of inversions. The case $q=1$ corresponds to the uniform distribution on $\bS_n$. A particularly interesting situation arises when $q = q_n \to 1$ as $n \to \infty$, capturing a limiting transition toward uniformity.

The following theorem was shown in \cite{Pinsky2022}
\begin{theorem}[Theorem 1.1 \cite{Pinsky2022}]\label{Thm:Pinsky}
	Let $q=q_n=1-c/n^{\alpha}$, for $c>0$, $\alpha\in (0,1)$, $\ell\ge 2$  be a constant integer, and
	$\varPi_n\sim\Mallow(n,q)$,
	\begin{align*}
		\frac{((\ell-1)!)^2\ell!}{(2\ell)!}\le &\liminf_{n\to\infty}{\frac{\bE \Arr_{\bS_\ell}(\varPi_n)}{n(1-q)^{\ell-1}}} &&\text{ and }& &\limsup_{n\to\infty}{\frac{\bE \Arr_{\bS_\ell}(\varPi_n)}{n(1-q)^{\ell-1}}}
		\le (\ell-1)!.
	\end{align*}
\end{theorem}
\begin{remark}\label{Rem:Thm:Pinsky}
	A careful analysis of the proof of Theorem~1.1 \cite{Pinsky2022} reveals that the arguments remain valid if the assumption concerning 
	$q$
	is replaced by 
	$(1-q)n/\ln n \to \infty$ as $n\to\infty$. 
\end{remark}
In this article, among others, we fill the gap from Theorem~\ref{Thm:Pinsky} and generalise the result to arbitrary $\cS$--arrangements. Namely we  establish the asymptotic value of  $\bE \Arr_{\cS}(\varPi_n)$, $\ell\ge 2$, for $\emptyset\neq\cS\subseteq \bS_\ell$, $\varPi_n\sim\Mallow(n,q)$, and $(1-q)n/\ln n \to \infty$ as $n\to\infty$. 	

In addition, in \cite{Pinsky2022} Pinsky notices that there is the value of $q$ for which there might be a kind of threshold for the property $\{\Arr_{\bS_\ell}(\varPi_n)\ge 1\}$, that there is a sequence of $\ell$ consecutive numbers in $\varPi_n\sim\Mallow(n,q)$. 
\begin{corollary}[Corollary 1.2 \cite{Pinsky2022}]\label{Cor:Pinsky}
	Let	$q=q_n=1-c/n^{\alpha}$, for $c>0$, $\alpha\in (0,1)$, \\
	$\varPi_n\sim\Mallow(n,q)$, $\ell\ge 3$  be a constant integer. Then\\
	\[\lim_{n\to\infty}\bE\Arr_{\bS_\ell}(\varPi_n)=
	\begin{cases}
		0&\text{ for }(1-q)n^{1/(\ell-1)}\to 0;\\
		\infty&\text{ for }(1-q)n^{1/(\ell-1)}\to \infty;
	\end{cases}	\]
\end{corollary}
In this article we study the distribution of $\Arr_\cS(\varPi_n)$, $\ell\ge 3$, $\varPi_n\sim\Mallow(n,q)$, near that threshold. More precisely, we show that, in a Mallows permutation $\varPi_n$, close to the threshold suggested by Corollary~\ref{Cor:Pinsky}, the random variable $\Arr_\cS(\varPi_n)$ is close in terms of the total variation distance to the random variable with the Poisson distribution $\operatorname{Po}(\bE (\Arr_\cS(\varPi_n)))$.

We use the following notations.
For $t_1\le t_2$, $t_1,t_2\in \NN_0=\{0,1,2,\ldots\}$ we define $[t_1,t_2]=\{t_1,t_1+1,\ldots,t_2\}$ and $[t_1,t_2]=\emptyset$ if $t_1>t_2$. Moreover $[t]=\{1,\ldots,t\}=[1,t]$, $t\in \NN$. We denote by $\bS_t$ the family of all permutations on $[t]$, $t\in \NN=\{1,2,3,\ldots\}$.
Moreover we use standard asymptotic notation as $n\to \infty$: $a_n=o(b_n)$ if $a_n/b_n\to 0$, $a_n=O(b_n)$ if $\exists_{C>0, N\in \NN}\forall_{n\ge N}|a_n|< Cb_n$, $a_n=\Theta(b_n)$ if $\exists_{c,C>0, N\in \NN}\forall_{n\ge N}cb_n<|a_n|< Cb_n$. All limits are taken as $n\to \infty$ and inequalities are true for large $n$.
We denote by $\operatorname{Geo}(\eta)$ and $\operatorname{Po}(\eta)$ the geometric and Poisson distribution with parameter $\eta$. For  an event $A$ by $\Ind_A$ we denote its indicator random variable.
Last but not least, for a sequence $(a_1,a_2,\ldots)$ and its two subsequences $(a_{i_1},\ldots,a_{i_b})$ and $(a_{j_1},\ldots,a_{j_c})$, $i_1<\ldots<i_b$, $j_1<\ldots<j_c$, we write $(a_{i_1},\ldots,a_{i_b})\subseq (a_{j_1},\ldots,a_{j_c})$ if $\{i_1,\ldots,i_b\}\subseteq\{j_1,\ldots,j_b\}$, i.e. when the first subsequence is a subsequence of the second one.

\section{Results}

In order to state our main result, we need one more definition. 
Let $Z_i\sim \operatorname{Geom}(1-q)$, $i=1,2,\ldots,\ell$, be independent random variables with the geometric distribution with parameter $1-q$. Them for $\emptyset\neq\cS\subseteq \bS_\ell$ we define
\begin{equation}\label{Eq:Def:LambdaS}\productO=\sum_{\sigma\in \cS}\prod_{j=1}^{\ell}\Pra{Z_j=\sigma(j)-|\sigma([j-1])\cap [\sigma(j)-1]|}.\end{equation}
This formula has a natural justification, which will be provided at the beginning of Section~\ref{Sec:Proof:Infinite} (see \eqref{Eq:Def:LambdaBis} and related discussion). Moreover we will prove in \eqref{Eq:Def:LambdaEll} that  
\begin{equation}\label{Eq:Def:Product}
	\product(\bS_\ell)=\prod_{j=1}^{\ell}(1-q^j).
\end{equation}

Now we are ready to formulate the first result of this article.

\begin{theorem}\label{Thm:Main:Expected}
	Let $\ell\ge 2$  be a constant integer, $\emptyset\neq \cS\subseteq \bS_\ell$, $\varPi_n\sim \Mallow(n,q)$, and
	$q=q(n)\in (0,1)$ be bounded away from $0$ by a constant and such that \[(1-q)n/\ln n\to \infty,\quad\text{ as }n\to\infty.\]
	Then
	\[
	\bE \Arr_{\cS}(\varPi_n) = (1+o(1)) \frac{\productO}{1-q^\ell}\frac{\prod_{i=1}^{\ell-1}(1-q^i)}{\prod_{i=1}^{\ell-1}(1-q^{\ell+i})}n.
	\]
	If moreover $1-q=o(1)$ then
	\[
	\bE \Arr_{\cS}(\varPi_n) = (1+o(1)) |\cS|\frac{[(\ell-1)!]^2}{(2\ell-1)!}n(1-q)^{\ell-1}. 
	\]	
\end{theorem}
Therefore by \eqref{Eq:Def:Product} we get the following corollary that improves the result from Theorem~\ref{Thm:Pinsky} (Theorem 1.1 \cite{Pinsky2022}). 
\begin{corollary}
	Let $\ell\ge 2$  be a constant integer, $\varPi_n\sim \Mallow(n,q)$, and
	$q=q(n)\in (0,1)$ be such that $1-q=o(1)$ and $(1-q)n/\ln n\to \infty$ as $n\to\infty$. Then
	\[
	\bE\Arr_{\bS_\ell}(\varPi_n) = (1+o(1)) \frac{\ell![(\ell-1)!]^2}{(2\ell-1)!}n(1-q)^{\ell-1}. 
	\]	
\end{corollary}

In the next result we analyse the distribution of $\Arr_{\cS}(\varPi_n)$ for $\ell\ge 3$ and values of $q$ such that $n(1-q)^{\ell-1}
$ is close to the threshold suggested by Corollary~\ref{Cor:Pinsky} (Corollary 1.2 \cite{Pinsky2022}). We find a range of $q$ for which the distribution of $\Arr_{\cS}(\varPi_n)$ is asymptotically Poisson.
For that we use the notion of the total variation distance. For random variables $X$ and $Y$ with non--negative integer values we have
\[\dtv(X,Y)=\max_{A\subseteq \NN_0}|\Pra{X\in A}-\Pra{Y\in A}|=\frac{1}{2}\sum_{k=0}^{\infty}|\Pra{X=k}-\Pra{Y=k}|.\] 
\begin{theorem}\label{Thm:Main:Poisson}
	Let $\ell\ge 3$ be a constant integer, $\emptyset\neq \cS\subseteq \bS_\ell$, $\varPi_n\sim \Mallow(n,q)$, and
	$q=q(n)\in (0,1)$ be such that 
	\begin{equation}\label{Eq:Poisson:q:condition}
		n^{-\frac{3}{2(\ell-1)}}=o(1-q)\quad\text{and}\quad 1-q=o\left((\ln n)^{-\frac{1}{\ell-2}}\right).
	\end{equation}
	Then, for $\lambda_n=\bE \Arr_\cS(\varPi_n)$ we have
	\[\dtv\left(\Arr_{\cS}(\varPi_n),\operatorname{Po}(\lambda_n)\right)=o(1).\]
\end{theorem}
Theorem~\ref{Thm:Main:Poisson}, Theorem~\ref{Thm:Main:Expected}, and properties of the total variation distance imply the following convergence result.
\begin{corollary}
	Let $\ell\ge 3$ be a constant, $\emptyset\neq \cS\subseteq \bS_\ell$, $\varPi_n\sim \Mallow(n,q)$, and
	$q=q(n)\in (0,1)$ be such that \[n(1-q)^{\ell-1}\to c,\quad\text{ as }n\to\infty,\]
	for some constant $c>0$. Then
	\[\Arr_{\cS}(\varPi_n) \stackrel{d}{\rightarrow} \operatorname{Po}\left(\frac{|\cS|[(\ell-1)!]^2}{(2\ell-1)!}c\right),\]
	where $d$ stands for the convergence in distribution.
\end{corollary}

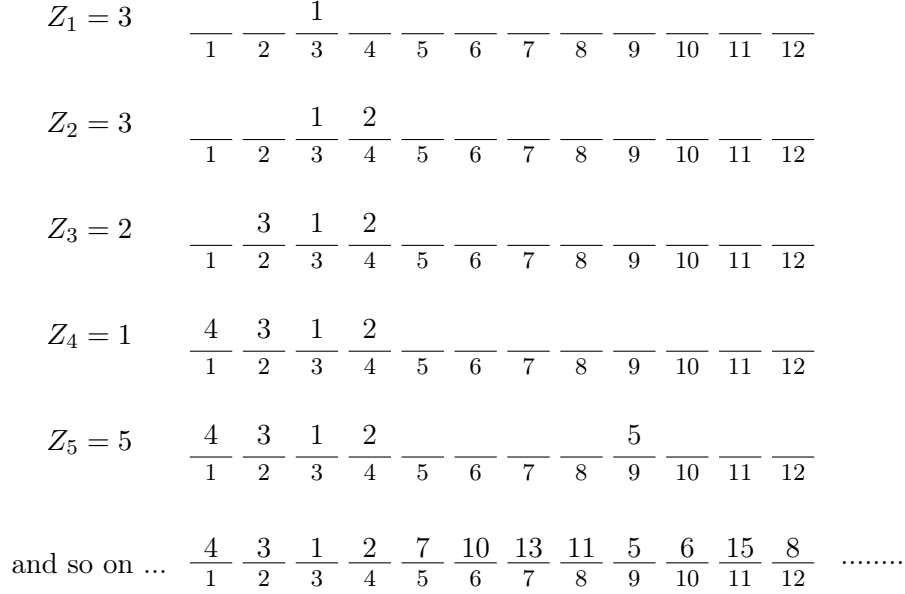
\begin{figure}[htbp]
	\centering
	\begin{tikzpicture}[xscale=.7]
		\foreach \j in {0,...,5}
		{
			\foreach \i in {1,...,12}
			{
				\draw (\i-0.4,1.4*\j)--(\i+0.4,1.4*\j);
				\node at (\i,1.4*\j-0.2) {{\footnotesize \i}};
		}}
		\foreach \i/\j in {1/3,2/4,3/2,4/1,5/9}
		{\foreach \k in {\i,...,5}
			\node at (\j,1.4*6-1.4*\k+.3) {\i};
		}
		\foreach \i/\j in {1/3,2/3,3/2,4/1,5/5}
		{
			\node at (-1.3,1.4*6-1.4*\i+0.2) {$Z_{\i}={\j}$};
		}
		\foreach \i/\j in {1/3,2/4,3/2,4/1,5/9,6/10,7/5,8/12,10/6,11/8,13/7,15/11}
		{\node at (\j,0.2) {\i};}
		\node at (-1.3,0) {and so on ...};
		\node at (13.5,0) {........};
	\end{tikzpicture}
	\caption{\label{Fig:Def:Mallows} Example of a Mallows infinite sequence $\barPi$.}
\end{figure}

In the proofs we will use an extension of the Mallows model to a random function ${\mathbf \Pi} : \bN\to\bN$, called the 
{\em Mallows process} and we will write ${\mathbf \Pi}\sim \Mallow(\NN,q)$ to denote its distribution with parameter $q\in (0,1)$. The process was considered by Gnedin and Olshanski \cite{GnedinOlshanski2010,GnedinOlshanski2012} and further utilised, among others, by Basu  and Bhatnagar  \cite{BasuBhat17} and He, M{\"{u}}ller, and Verstraaten \cite{HeMV23}. It will be more convenient for us to analyse the related sequences of preimages of the Mallows process. We will denote the random sequences related to ${\mathbf \Pi}\sim \Mallow(\NN,q)$ by 
\[\barPi=(\barPi_1,\barPi_2,\ldots)=({\mathbf \Pi}^{-1}(1),{\mathbf \Pi}^{-1}(2),\ldots).\]
We will call it the {\sl Mallows infinite sequence} and denote its distribution by $\Mallows(\NN,q)$.
The sequence $\barPi\sim \Mallows(\NN,q)$ is constructed in a recursive manner (see Figure~\ref{Fig:Def:Mallows}). 

We label the spots on a horizontal line by the consecutive natural numbers. Then we form a random sequence $\barPi$ in the following manner.  
\begin{itemize}
	\item[] For all $i\in \NN$, when numbers from  $[i-1]$ have been placed on the line, we sample $Z_i\sim\operatorname{Geo}(1-q)$ according to the geometric distribution, independently of previous $Z_1,\ldots,Z_{i-1}$. If $Z_i=k$ then we place $i$ in the $k$-th unoccupied spot on the line.
\end{itemize}
The Mallows process  ${\mathbf \Pi}\sim\Mallow(\NN,q)$, ${\mathbf \Pi} : \bN\to\bN$, related to the Mallows infinite sequence $\barPi\sim\Mallows(\NN,q)$ is given by the formula ${\mathbf \Pi}(i)=j$ iff $\barPi_j=i$, where $\barPi_j$ is the $j$--th term of the sequence $\barPi$. 
For the Mallows infinite process ${\mathbf \Pi}\sim\Mallow(\bN,q)$ and sequence $\barPi\sim \Mallows(\NN,q)$ we define $(\cS,k_1,k_2)$--arrangement in an analogous manner as for $\varPi_n\sim \Mallow(n,q)$ (see Definition~\ref{Def:Sarr}).
\begin{definition}\label{Def:Sarr:barPi}
	Let $\ell\ge 2$, $n\in\bN$, $\emptyset\neq\cS\subseteq \bS_\ell$ and $k_1$, $k_2$ be nonnegative integers. We say that there is an $(\cS,k_1,k_2)$--arrangement in ${\mathbf \Pi}:\bN\to \bN$ if
	\begin{equation*}
		\exists_{\sigma\in\cS}\forall_{i\in [\ell]} {\mathbf \Pi}(k_2+i)=k_1+\sigma(i).
	\end{equation*}
	i.e. there is na $(\cS,k_1,k_2)$--arrangement in $\barPi:\bN\to\bN$ if
	\begin{equation}\label{Eq:SArr:barPi}
		\exists_{\sigma\in\cS}\forall_{i\in [\ell]} \barPi_{k_1+i}=k_2+\sigma^{-1}(i).
	\end{equation}
	We call any  $(\cS,k_1,k_2)$--arrangement in ${\mathbf \Pi}$ or $\barPi$ an  $\cS$--arrangement.
\end{definition}

For  $1\le t_1 \le t_2$ we define
\begin{equation}\label{Eq:Def:barPi:Obciety}
	\barPi|_{[t_1,t_2]}=(\barPi_{t_1},\barPi_{t_1+1},\ldots,\barPi_{t_2}).
\end{equation}
In other words the subsequence $\barPi|_{[t_1,t_2]}$ consists of terms of $\barPi$ with indices from $t_1$ to $t_2$. 
Denote by $\Arr_{\cS}(\barPi|_{[t_1,t_2]})$ the number of $\cS$--arrangements in $\barPi|_{[t_1,t_2]}$.

The following theorems will be the main ingredient of the proof of Theorem~\ref{Thm:Main:Expected}.

\begin{theorem}\label{Thm:ExpectedValueProcess}
	Let $\ell\ge 2$ be a constant, $\emptyset\neq\cS\subseteq \bS_\ell$, and $q=q(n)\in (0,1)$ and $t=t(n)\in \NN$ be such that $(1-q)t/\ln n\to \infty$ and $t\to\infty$ as $n\to\infty$. Then
	\[
	\bE \Arr_{\cS}(\barPi|_{[t]}) =(1+o(1)) \frac{\productO}{1-q^\ell}\frac{\prod_{i=1}^{\ell-1}(1-q^i)}{\prod_{i=1}^{\ell-1}(1-q^{\ell+i})}t,
	\]
	where $\productO$ is defined in \eqref{Eq:Def:Product}.
	
	Moreover if $1-q=o(1)$ then 
	\[
	\bE \Arr_{\cS}(\barPi|_{[t]}) =(1+o(1)) |\cS|(1-q)^{\ell-1}\frac{[(\ell-1)!]^2}{(2\ell-1)!}t. 
	\]
\end{theorem}
Let
\[\bB_{k}=\Pra{\exists_{j\in\NN_0 }\text{ there is an $(\cS,j,k)$--arranegement in $\barPi$}}.\]
In the proof of Theorem~\ref{Thm:ExpectedValueProcess} we will use the following result.
\begin{theorem}\label{Thm:Qk:value}
	Let $\ell\ge 2$ be a constant, $\emptyset\neq\cS\subseteq \bS_\ell$, $q=q(n)\in (0,1)$, and $t=t(n)\in \NN$ be such that $(1-q)t/\ln n\to \infty$. Then for $k=k(n)\ge \ell -\log_q n$
	\begin{equation*}
		\begin{split}
			\bB_{k}
			=(1+O(n^{-1}))
			\frac{\productO}{1-q^\ell}
			\prod_{i=1}^{\ell-1}\frac{(1-q^i)}{(1-q^{\ell+i})}+O(\productO n^{-\ell}),
		\end{split}
	\end{equation*}
	where $\productO$ is defined in \eqref{Eq:Def:Product}.
	
	Moreover for $1-q=o(1)$ we have
	\begin{equation*}
		\begin{split}
			\bB_{k}=(1+O(n^{-1}+O(1-q)))\frac{\productO}{1-q}\frac{(\ell-1)!}{(2\ell-1)!}
			=(1+o(1))\frac{|\cS|[(\ell-1)!]^2}{(2\ell-1)!}(1-q)^{\ell-1}.
		\end{split}
	\end{equation*}
\end{theorem}

\section{Relations between Mallows permutations and the Mallows infinite process}

The results in this section are crucial to the proofs of Theorem~\ref{Thm:Main:Expected} and Theorem~\ref{Thm:Main:Poisson} since they serve as a link between the results concerning $\barPi\sim\Mallows(\NN,q)$  and those concerning $\varPi_n\sim\Mallow(n,q)$.

\subsection{Mallows sequence}

Due to the method of the proof, it will be convenient to consider,   instead the Mallows permutation $\varPi_n\sim \Mallow(n,q)$, its sequence of preimages, that is 
\[\barpi_n := (\barpi_{n,1},\ldots,\barpi_{n,n}),\text{ where }\ \barpi_{n,i}=\varPi_n^{-1}(i),\ i\in [n],\text{ for }\varPi_n\sim \Mallow(n,q).\]
We will denote by
$\Mallows(n,q)
$
the distribution of the sequence of preimages of $\varPi_n\sim \Mallow(n,q)$, and we will write  $\barpi_n\sim \Mallows(n,q)$. Note that,  there is a natural, straight forward one--to--one correspondence between $\cS$--arrangements  in $\varPi_n\sim \Mallow(n,q)$ and arrangements of consecutive elements  in $\barpi_n\sim\Mallows(n,q)$.  
\begin{remark}
	For $\emptyset\neq \cS\subseteq \bS_\ell$, $0\le k_1\le n-\ell$, $0\le k_2\le n-\ell$, there is an $(\cS,k_1,k_2)$--arrangement in a permutation $\varPi_n: [n]\to [n]$ (see Definition~\ref{Def:Sarr}) if and only if 
	\begin{equation}\label{Def:Sarr:BarPi}
		\exists_{\sigma\in\cS}\forall_{i\in [\ell]} \barpi_{n,k_1+i}=k_2+\sigma^{-1}(i).
	\end{equation}	
	See Definition~\ref{Def:Sarr:barPi} for comparison.
	We will say that there is an $(\cS,k_1,k_2)$--arrangement in $\barpi_n$ if \eqref{Def:Sarr:BarPi} is fulfilled. For any $0\le k_1\le n-\ell$, $0\le k_2\le n-\ell$, we call any $(\cS,k_1,k_2)$--arrangement in $\barpi_n$ an  $\cS$--arrangement and we denote  the number of $\cS$--arrangements in $\barpi_n$ by $\Arr_{\cS}(\barpi_n)$.  
\end{remark}
This remarks imply the following fact
\begin{fact}\label{Rem:Perm:Seq}
	The statements of Theorem~\ref{Thm:Main:Expected} and~Theorem~\ref{Thm:Main:Poisson} are equivalent to the ones with $\varPi_n\sim\Mallow(n,q)$ replaced by $\barpi_n\sim\Mallows(n,q)$.
\end{fact}

\subsection{Mallows infinite process subsequences}

In this subsection we will study relations between the $\cS$--arrangements in some subsequences of $\barPi\sim\Mallows(\bN,q)$ and $\cS$--arrangements in $\barpi_t\sim\Mallows(t,q)$, for various $t$. 
First we define a family of subsequences of  $\barPi\sim\Mallows(\NN,q)$. 

\begin{definition}\label{Def:barPi:Obraz}
	Let    $a_1<a_2<\ldots<a_{t_2-t_1+1}$ be indices of all the terms of the sequence $\barPi$ that have values in the interval $[t_1,t_2]$. Then we set
	\begin{equation*}
		\barPi([t_1,t_2])=
		(\barPi_{a_1},\barPi_{a_2},\ldots,\barPi_{a_{t_2-t_1+1}}),
	\end{equation*}	
	and 
	\begin{equation*}
		\barbarPi([t_1,t_2])=(\barpi_{[t_1,t_2],1},\ldots,\barpi_{[t_1,t_2],t_2-t_1+1}),\text{ where }\forall_{i\in [t_2-t_1+1]}\ \barpi_{[t_1,t_2],i}=\barPi_{a_i}-(t_1-1). 
	\end{equation*}	
\end{definition}
In other words, $\barPi([t_1,t_2])$ is the subsequence of $\barPi$ consisting of those terms of $\barPi$ that have values between $t_1$ and $t_2$. Moreover $\barbarPi([t_1,t_2])$ is a sequence of preimages of a permutation in $\bS_{t_2-t_1+1}$ that is build upon $\barPi([t_1,t_2])$. It retains the order of elements of $\barPi([t_1,t_2])$ but rescales so that the values of terms of $\barbarPi([t_1,t_2])$ would be in $[t_2-t_1+1]$ instead of $[t_1,t_2]$. For example, in Figure~\ref{Fig:Def:Mallows} we have $\barPi([4,8])=(4,7,5,6,8)$ and 
$\barbarPi([4,8])=(\barpi_{[4,8],1},\barpi_{[4,8],2},\ldots,\barpi_{[4,8],5})=(4-3,7-3,5-3,6-3,8-3)=(1,4,2,3,5)$. Note that rescaling preserves $\cS$--arrangements, therefore there is one to one correspondence between $\cS$--arrangements in $\barPi([t_1,t_2])$ and $\barbarPi([t_1,t_2])$. Last but not least $\barPi([t])=\barbarPi([t])$.

The following lemma states a relation between the distributions of some subsequences of $\barPi\sim \Mallows(\NN,q)$ and the distribution of $\barpi_t\sim\Mallow(t,q)$, $t\in\NN$. We state it in the way convenient for our purposes.

\begin{lemma}[\cite{GnedinOlshanski2010}, see also {\cite[Lemma~2.1]{BasuBhat17}} ]\label{Lem:CoupligProcesses}
	Let $q\in (0,1)$ and
	$\barPi\sim\Mallows(\NN,q)$. Then for any $0=t_0<t_1<t_2<\ldots<t_s$, $\barbarPi([t_{i-1}+1,t_i])$, $i\in [s]$, are independent and  
	\[
	\barbarPi([t_{i-1}+1,t_i])\sim \Mallows(t_i-t_{i-1},q),\quad i\in [s].
	\]
\end{lemma}
The lemma together with Fact~\ref{Rem:Perm:Seq} imply that the problem of finding the bound on $\Arr_{\cS}(\varPi_n)$, where $\varPi_n\sim \Mallow(n,q)$, is reduced to analysing the value of $\Arr_{\cS}(\barPi([n]))$ (recall that $\barPi([n])=\barbarPi([n])$).
\begin{corollary}\label{Cor:varPi:to:barPi}
	The statements of Theorem~\ref{Thm:Main:Expected} and~Theorem~\ref{Thm:Main:Poisson} are equivalent to the ones in which $\varPi_n\sim\Mallow(n,q)$ is replaced by $\barPi([n])$, $\barPi\sim\Mallows(\NN,q)$.	
\end{corollary}
In the remaining part of the paper, we will focus on the value of $\Arr_{\cS}(\barPi([n]))$.

Recall that the subsequence $\barPi|_{[t_1,t_2]}$ consists of the terms of $\barPi$ with indices from $t_1$ to $t_2$.
To illustrate the following lemma we use example in Figure~\ref{Fig:Sarr}. In the Mallows sequence in  Figure~\ref{Fig:Sarr} 
\[\barPi([7])\stackrel{(a1)}{\subseq} \barPi|_{[10]}\stackrel{(a2)}{\subseq} \barPi([15]),\]
where $\subseq$ is defined at the end of Introduction.
\begin{figure}[t]
	\centering
	\begin{tikzpicture}[xscale=.6]
		\foreach \i/\j in {2/1,2/2,2/3,6/1}
		{
			\draw[rounded corners,fill=lightgray, draw=black]  (\i-0.35,1.4*\j+0.05) rectangle (\i+2.35,1.4*\j+0.6);
		}
		\foreach \j in {2,3}
		{\draw[rounded corners,fill=lightgray, draw=black]  (5-0.35,1.4*\j+0.05) rectangle (8+2.35,1.4*\j+0.6);}
		\draw[rounded corners,fill=lightgray, draw=black]  (12-0.35,1.4*1+0.05) rectangle (13+2.35,1.4*1+0.6);
		\foreach \j in {1,2,3}
		{
			\foreach \i in {1,...,19}
			{
				\draw (\i-0.4,1.4*\j)--(\i+0.4,1.4*\j);
				\node at (\i,1.4*\j-0.2) {{\footnotesize \i}};
		}}
		\foreach \i/\j in {1/4,2/3,3/1,4/2,5/7,6/14,7/12,8/13,9/5,10/6,12/11,13/9,15/10,17/8,18/15}
		{
			\node at (\i,1.4+.3) {\j};
		}
		\foreach \i/\j in {1/4,2/3,3/1,4/2,5/7,9/5,10/6,13/9,15/10,17/8}
		{
			\node at (\i,2*1.4+.3) {\j};
		}
		\foreach \i/\j in {1/4,2/3,3/1,4/2,5/7,9/5,10/6}
		{
			\node at (\i,3*1.4+.3) {\j};
		}
		\foreach \i/\j in {3/7,2/10,1/15}
		{
			\node at (-2.6,1.4*\i+0.2) {$\barPi([\j])=(\barpi_{\j,i})_{i\in [\j]}$};
		}
	\end{tikzpicture}
	\caption{\label{Fig:Sarr} Examples of the Mallows sequences $\barpi_{t}$ in the coupling with the Mallows infinite sequence $\barPi$. In grey rectangles $\cS$--arrangements in $\varPi_t$ for $\cS=\{\sigma\}$, $\sigma:[3]\to [3]$, $\sigma(1)=2$, $\sigma(2)=3$, $\sigma(3)=1$, i.e. $(\sigma^{-1}(1),\sigma^{-1}(2),\sigma^{-1}(3))=(3,1,2)$.}
\end{figure}
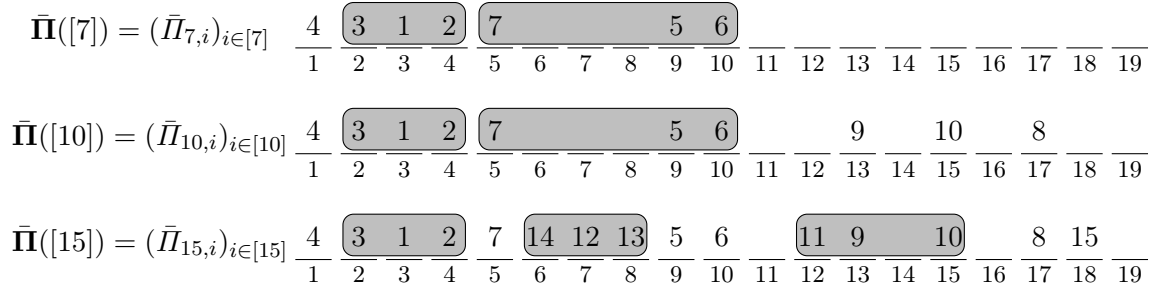

We make the following observations.
\begin{itemize}
	\item[(i)] By  $(a2)$ all $(\cS,k_1,k_2)$--arrangements in $\barPi|_{[10]}$, namely $(3,1,2)$ and $(14,12,13)$, are $(\cS,k_1,k_2)$--arrangements with $k_1\le 10-3=7$ in $\barPi([15])$. However $\barPi([15])$ might have more $\cS$--arrangements (in the example it is $(11,9,10)$). 
	\item[(ii)] By $(a1)$ and $(a2)$ all $(\cS,k_1,k_2)$--arrangements in $\barPi([15])$ with $k_2\le 7-3=4$, namely $(3,1,2)$, are $(\cS,k_1,k_2)$--arrangements in $\barPi|_{[10]}$. Note that $\barPi|_{[10]}$ may have other $(\cS,k_1,k_2)$--arrangements of $\barPi([15])$ with $k_2 > 7-3=4$, such as $(14,12,13)$. 
	\item[(iii)] The sequence $\barbarPi([7-3+1,15])=(3,10,8,9,1,2,7,5,6,4,11)$ is obtained from $\barPi([7-3+1,15])=(7,14,12,13,5,6,11,9,10,8,15)$ by rescaling. The first one has $\cS$--arrangements $(10,8,9)$ and $(7,5,6)$. These $\cS$--arrangements correspond to $(14,12,13)$  and  $(11,9,10)$ in $\barPi([15])$. In fact each $(\cS,k_1,k_2)$--arrangement in $\barPi([15])$ with $k_2 > 7-3$ corresponds to an $\cS$--arrangement in $\barbarPi([7-3+1,15])$. 
\end{itemize}

\begin{lemma}\label{Lem:Inclusion:Implication}
	Let $1\le t_1\le t_2\le t_3$ and $\emptyset\neq\cS\subseteq \bS_\ell$.
	\begin{align}
		\label{Eq:Inclusion:Implication1}
		\barPi|_{[t_2]}\subseq \barPi([t_3]) &\Rightarrow	\Arr_{\cS}(\barPi|_{[t_2]})\le \Arr_{\cS}(\barPi([t_3])) \\
		\label{Eq:Inclusion:Implication2}
		\barPi([t_1])\subseq \barPi|_{[t_2]}\subseq \barPi([t_3]) &\Rightarrow	
		\Arr_{\cS}(\barPi([t_3]))\le \Arr_{\cS}(\barPi|_{[t_2]}) +  \Arr_{\cS}(\barbarPi([t_1-\ell+1,t_3]))
	\end{align}	
\end{lemma}

\begin{proof}[Proof of Lemma~\ref{Lem:Inclusion:Implication}]
	Let $1\le t_1\le t_2\le t_3$ and set (see Definition~\ref{Def:barPi:Obraz})
	\[
	\barPi([t_1])=(\barPi_{a_1},\ldots,\barPi_{a_{t_1}})\quad\text{and}\quad
	\barPi([t_2])=(\barPi_{b_1},\ldots,\barPi_{b_{t_2}}).
	\]
	Then 
	\[\barPi([t_1])\subseq \barPi|_{[t_2]}, \text{ is }(\barPi_{a_1},\ldots,\barPi_{a_{t_1}})\subseq (\barPi_1,\ldots,\barPi_{t_2}).\]
	Moreover
	\[
	\barPi|_{[t_2]}\subseq\barPi([t_3]) \text{ means that } \forall_{i\in [t_2]} \barPi_{a_i}=\barPi_i\quad (a_i=i).\]
	Therefore for any  $\emptyset\neq\cS\subseteq \bS_\ell$.
	\begin{itemize}
		\item[(i)] If $\barPi|_{[t_2]}\subseq \barPi([t_3])$, then each $(\cS,k_1,k_2)$--arrangement in $\barPi$ with $0\le k_1\le t_2-\ell$ (i.e. $\cS$--arrangements in $\barPi|_{[t_2]}$) is an $(\cS,k_1,k_2)$--arrangement in $\barPi([t_3])$. 
		\item[(ii)] If $\barPi([t_1])\subseq \barPi|_{[t_2]}\subseq \barPi([t_3])$ then each $(\cS,k_1,k_2)$--arrangement in $\barPi([t_3])$ with $k_2 \le  t_1-\ell$  is an $(\cS,k_1,k_2)$--arrangement in $\barPi|_{[t_2]}$ (since all values from $[t_1]$ are in $t_2$ first terms of both $\barPi$ and $\barPi([t_3])$).
		\item[(iii)] Each $(\cS,k_1,k_2)$--arrangement in $\barPi([t_3])$ with $k_2 > t_1-\ell$ correspond to an $(\cS,k_1',k_2-t_1+\ell)$--arrangement in $\barbarPi([t_1-\ell+1,t_3])$ (see Definition~\ref{Def:barPi:Obraz}), for some $k_1'$ depending on $\barPi$.
	\end{itemize}

	Finally (i) implies \eqref{Eq:Inclusion:Implication1}
	and (ii) with (iii) imply \eqref{Eq:Inclusion:Implication2}.
\end{proof}

Define
\begin{equation}
	\label{Eq:Def:Delta}	
	\Delta=\Bigg\lceil \frac{\ln n}{1-q}\Bigg\rceil.
\end{equation}
\begin{lemma}\label{Lem:Inclusion0}
	Let $q\in (0,1)$ and $\barPi\sim\Mallows(\NN,q)$. 
	If $n(1-q)\ge 1$ and $t\ge x\Delta$ then for any $x>0$ 	
	\[
	\Pra{\barPi|_{[t-x\Delta]}\subseq \barPi([t])}\ge 1-2n^{-x+1}.
	\]
\end{lemma}
\begin{lemma}\label{Lem:Inclusion2}
	Let $q\in (0,1)$ and $\barPi\sim\Mallows(\NN,q)$. 
	If $n(1-q)\ge 1$ and $t\ge 1$ then for any $x>0$ 	
	\[\Pra{\barPi([t])\subseq \barPi|_{[t+x\Delta]}}\ge 1-n^{-x+1}.
	\]
\end{lemma}

In the proof of Lemma~\ref{Lem:Inclusion0} we will use the following result from \cite{BhatnagarPeled2015}.

\begin{lemma}[{\cite[Theorem 1.1]{BhatnagarPeled2015}}]\label{Lem:CouplingKnown}
	Let $\varPi_t\sim\Mallow(t,q)$, $q\in (0,1)$, $t\in \NN$, then
	\[
	\forall_{x\ge 1}\forall_{i\in [t]}\Pra{|\varPi_t(i)-i|\ge x}\le 2q^{x}. 
	\]
\end{lemma}

\begin{proof}[Proof of Lemma~\ref{Lem:Inclusion0}]
	Recall that for each $k\ge 1$ we defined
	\[\barPi([k])=\barbarPi([k])=(\barPi_{a_{k,1}},\ldots,\barPi_{a_{k,k}})
	\]
	to be the subsequence of $\barPi\sim \Mallows(\bN,q)$ that consists of the terms of $\barPi$ with values in $\{1,\ldots,k\}$ (Definition~\ref{Def:barPi:Obraz}).
	By Lemma~\ref{Lem:CoupligProcesses} we have that, for all $k\in \NN$, $\barPi([k])$ has the distribution of the sequence of preimages $(\barpi_{k,1},\ldots,\barpi_{k,k})$ of  $\varPi_{k}\sim\Mallow(k,q)$. Therefore
	\[\forall_{i\in [k]} \barPi_{a_{k,i}}=\barpi_{k,i}=\varPi_k^{-1}(i).\]
	Thus
	\begin{equation}\label{Eq:Incl:x1}
		\forall_{i\in [k]} \varPi_k(\barPi_{a_{k,i}})=i.	
	\end{equation}
	Now let us analyse the event  $\{\barPi|_{[t-x\Delta]}\subseq \barPi([t])\}.$
	$\{\barPi|_{[t-x\Delta]}\subseq \barPi([t])\}$ occurs iff for all $j\in [t-x\Delta]$ we have $\barPi_j\le t$. Therefore the opposite event to $\{\barPi|_{[t-x\Delta]}\subseq \barPi([t])\}$  is the the event that there exists $k$, $k>t$, such that for some $j\in [t-x\Delta]$ we have $\barPi_j=k$.

	Now let us assume that, given $k>t$, for some $j\in [t-x\Delta]$ we have $\barPi_j=k$. Then in $\barPi([k])$ we have for some $i\in [k]$ 
	\begin{equation}\label{Eq:Incl:x2}
		\barPi_{a_{k,i}}=k\quad\text{ and }\quad a_{k,i}=j\le t-x\Delta.
	\end{equation}
	Note that $a_{k,1}<a_{k,2}<\ldots<a_{k,k}$. Therefore $i\le a_{k,i}$, for all $i\in [k]$. Thus by \eqref{Eq:Incl:x1} and \eqref{Eq:Incl:x2}
	\begin{equation}\label{Eq:Incl:x3}
		\varPi_k(k)=\varPi_k(\barPi_{a_{k,i}})=i\le a_{k,i} \le t- x\Delta.
	\end{equation}
	
	To summarize the discussion leading to \eqref{Eq:Incl:x3}, if event $\{\barPi|_{[t-x\Delta]}\subseq \barPi([t])\}$ does not occur, then there exists $k$, $k>t$, such that
	\[	\varPi_k(k)\le t- x\Delta.
	\]	
	Therefore for $n(1-q)\ge 1$, using Lemma~\ref{Lem:CouplingKnown} (in the fourth inequality), we obtain
	\begin{align*}
		1-\Pra{\barPi|_{[t-x\Delta]}\subseq \barPi([t])}
		&\le \Pra{\exists_{k > t}\varPi_k(k)\le k - (k-t) - x\Delta}\\
		&\le \sum_{k=t+1}^{\infty}\Pra{\varPi_k(k)\le k - (k-t) - x\Delta}\\
		&\le \sum_{k=t+1}^{\infty}\Pra{|\varPi_k(k)-k|\ge (k-t) + x\Delta}\\
		&\le \sum_{k=t}^{\infty} 2q^{x\Delta}q^{k-t}\\
		&\le 2[1-(1-q)]^{x\ln n/(1-q)}\frac{1}{1-q}\\
		&\le 2n^{-x+1}.
	\end{align*}
	Thus
	\[
	\Pra{\barPi|_{[t-x\Delta]}\subseq \barPi([t])}\ge 1-2n^{-x+1}.
	\]
\end{proof}

\begin{proof}[Proof of Lemma~\ref{Lem:Inclusion2}] 
	In the construction of the Mallows infinite sequence we have that event $\{\barPi([t])\subseq \barPi|_{[t+x\Delta]}\}$ occurs iff for all $i\in [t]$ we have $Z_i\le t+x\Delta-(i-1)$. Therefore, since $Z_i\sim \operatorname{Geo}(1-q)$ and $n(1-q)\ge 1$,
	\begin{align*}
		1-&\Pra{\{\barPi([t])\subseq \barPi|_{[t+x\Delta]}\}}
		=\Pra{\exists_{i\in [t]}Z_i > t+x\Delta-(i-1)}\le \sum_{i=1}^{t}\Pra{Z_i > t+x\Delta-(i-1)}\\
		&= \sum_{i=1}^{t}q^{x\Delta+(t-i+1)}
		\le \sum_{j=x\Delta}^{\infty} q^j=\frac{q^{x\Delta}}{1-q}
		\le \frac{(1-(1-q))^{x(1-q)^{-1}\ln n}}{1-q}
		\le ne^{-x\ln n}\le n^{-x+1}. 
	\end{align*}
\end{proof}

\section{Proof of Theorem~\ref{Thm:ExpectedValueProcess} and Theorem~\ref{Thm:Qk:value}}\label{Sec:Proof:Infinite}

In this section $\ell\ge 2$ is a constant integer, 
and $\cS$, $\emptyset\neq\cS\subseteq \bS_\ell$, is a family of permutations. For $k\ge 0$ and $j\ge 0$, we define
\[
\begin{split}
	\cA[k,j] &\text{ -- event that $\barPi$ contains an $(\cS,k,j)$--arrangement};\\
	\cB[k] &\text{ -- event that, for some $j\ge 0$, $\barPi$ contains an $(\cS,k,j)$--arrangement}.
\end{split}	
\]
Moreover let
\[
\bB_{k}=\Pra{\cB[k]}.
\]
be the probability that $(\barPi_{k+1},\ldots,\barPi_{k+\ell})$ forms an $\cS$--arrangement in $\barPi$, as defined in Definition~\ref{Def:Sarr:BarPi}.

Let us note that 
\begin{equation}\label{Eq:Def:LambdaBis}
	\Pra{\cA[0,0]}=\sum_{\sigma\in \cS}\prod_{j=1}^{\ell}\Pra{Z_j=\sigma(j)-|\sigma([j-1])\cap [\sigma(j)-1]|}=\productO
\end{equation}
is the probability that $(\barPi_1,\ldots,\barPi_\ell)$ is an $\cS$--arrangement consisting of values from $[\ell]$. The last equality follows by \eqref{Eq:Def:LambdaS}.
In the special case $\cS=\bS_\ell$, for the event $\cA[0,0]$ to occur, we require only that in the Mallows infinite process ${\mathbf \Pi}(i)\le \ell$ for all $i\in [\ell]$, i.e. $Z_i\le \ell+1-i$ for all $i\in [\ell]$. Therefore we have  
\begin{equation}\label{Eq:Def:LambdaEll}
	\product(\bS_\ell)=\prod_{i=1}^{\ell}\Pra{Z_i\le \ell+i-1}=(1-q^{\ell})(1-q^{\ell-1})\ldots (1-q)=\prod_{i=1}^{\ell}(1-q^i).
\end{equation}.

\subsection{Technical lemmas}

In this section we prove two technical lemmas  that are necessary to obtain Theorem~\ref{Thm:Qk:value}.
For integers $0\leq i<j$ we define 
\begin{equation}\label{Eq:Def:SsI}
	s(\emptyset)=0,\ s(I)=\sum_{r\in I}r,\text{ for }\emptyset\neq I\subseteq [j-1],\quad\text{and}\quad \sumI^{(j)}_{i}=\sum_{I\subseteq [j-1],|I|=i}q^{s(I)}.
\end{equation}
Moreover we set
\begin{equation}
	\product[a,b]=\frac{\product(\bS_{b})}{\product(\bS_{a-1})}=\prod_{i=a}^{b}(1-q^{i}),
	\text{ for }1\le a\le b,
	\quad
	\text{ and }
	\quad
	\product[a,b]=1,
	\text{ for }a > b,
\end{equation}
where $\product(\bS_b)$ is defined in \eqref{Eq:Def:Product}.

\begin{lemma}\label{Lem:Qk:SumIProp}
	Let $q\in (0,1)$ and $\ell\geq 2$. Then, 
	\begin{equation*}
		\sum_{i=0}^{\ell-1}(-1)^{i}\frac{\sumI_{i}^{(\ell)}}{\product[\ell+i,\ell+i]}=\frac{\product[1,\ell-1]}{\product[\ell,2\ell-1]}=
		\frac{\prod_{i=1}^{\ell-1}(1-q^i)}{\prod_{i=0}^{\ell-1}(1-q^{\ell+i})}.
	\end{equation*}
\end{lemma}

\begin{lemma}\label{Lem:AuxiliarySum} Let $q =q(n)\in (0,1)$ and $k\ge \ell -\log_q n$. Then, 
	\[\sum_{j=0}^{k-\ell+1}\left(q^{\ell}\right)^j\prod_{r=1}^{\ell-1}(1-q^{j+r})
	=(1+O(n^{-\ell}))\frac{1}{1-q^\ell}\prod_{i=1}^{\ell-1}\frac{(1-q^i)}{(1-q^{\ell+i})}.\]
\end{lemma}

\begin{proof}[Proof of Lemma~\ref{Lem:Qk:SumIProp}]
	First we prove some auxiliary equalities
	\begin{equation}\label{Eq:Sumtricks}	\begin{aligned}
			\sumI^{(j)}_{0}&=1, &&\text{ for all }j\ge 1;\\
			\sumI^{(j)}_{i}&=q^{i}\,\sumI^{(j-1)}_{i-1} + q^{i}\,\sumI^{(j-1)}_{i}, &&\text{ for } 1 \le i \le j-2\\
			\sumI^{(i)}_{i-1}&=q^{i-1}\sumI^{(i-1)}_{i-2}, &&\text{ for }i\ge 2.
		\end{aligned} 
	\end{equation}
	The first one follows directly from the definition \eqref{Eq:Def:SsI} of $\sumI^{(j)}_{i}$, 
	\[
	\sumI^{(j)}_{0}=q^{s(\emptyset)}=1.
	\]
	To obtain the second equality, we divide each set $I$ into two subsets, based on whether or not they contain $1$, and get
	\begin{equation*}
		\begin{split}
			\sumI^{(j)}_{i}
			&=\sum_{I\subseteq [j-1],|I|=i, 1\in I}q^{1+\sum_{r\in I\setminus\{1\}}r} + \sum_{I\subseteq [j-1],|I|=i, 1\notin I}q^{\sum_{r\in I}r} \\
			&=\sum_{I\subseteq [j-1],|I|=i, 1\in I}q^{i+\sum_{r\in I\setminus\{1\}}(r-1)} + \sum_{I\subseteq [j-1],|I|=i, 1\notin I}q^{i+\sum_{r\in I}(r-1)}\\ 
			&=\sum_{I\subseteq [j-2],|I'|=i-1}q^iq^{\sum_{r'\in I'}r'} + \sum_{I'\subseteq [j-2],|I'|=i}q^iq^{\sum_{r'\in I'}r'}\\ 
			&=q^{i}\,\sumI^{(j-1)}_{i-1} + q^{i}\,\sumI^{(j-1)}_{i}.
		\end{split}
	\end{equation*}
	Similarly, the third equality follows from \eqref{Eq:Def:SsI}, as for $i\ge 2$
	\begin{equation*}
		\sumI^{(i)}_{i-1}=q^{i-1+\sum_{r\in [i-2]}r}=q^{i-1}\sumI^{(i-1)}_{i-2}.
	\end{equation*} 
	We now use the three identities from \eqref{Eq:Sumtricks} to prove the following relation for $j=0,1,\ldots,\ell-2$
	\begin{equation}\label{Eq:Qk:InductionStep}
		\sum_{i=0}^{(\ell-1-j)}(-1)^{i}q^{ji}\frac{\sumI_{i}^{(\ell-j)}}{\product[\ell+i,\ell+i+j]}
		=
		(1-q^{j+1})\sum_{i=0}^{\ell-1-(j+1)}(-1)^{i}q^{(j+1)i}\frac{\sumI_{i}^{(\ell-(j+1))}}{\product[\ell+i,\ell+i+(j+1)]}.
	\end{equation}
	Indeed, for $j=0,1,\ldots,\ell-2$, expanding the sum on the left-hand side of \eqref{Eq:Qk:InductionStep} and applying \eqref{Eq:Sumtricks} gives
	\begin{align*}
		\sum_{i=0}^{\ell-1-j}(-1)^{i}q^{ji}&\frac{\sumI_{i}^{(\ell-j)}}{\product[\ell+i,\ell+i+j]}\\
		&=(-1)^{0}q^{0}\frac{\sumI_{0}^{(\ell-(j+1))}}{\product[\ell,\ell+j]}
		+\sum_{i=1}^{\ell-1-(j+1)}(-1)^{i}q^{ji}\frac{q^{i}\,\sumI_{i}^{(\ell-(j+1))} + q^{i}\,\sumI_{i-1}^{(\ell-(j+1))}}{\product[\ell+i,\ell+i+j]}\\
		&\qquad+(-1)^{\ell-1-j}q^{j(\ell-1-j)}\frac{q^{\ell-1-j}\sumI_{\ell-1-(j+1)}^{(\ell-(j+1))}}{\product[2\ell-1-j,2\ell-1]}
		\\
		&=\sum_{i=0}^{\ell-1-(j+1)}(-1)^{i}q^{ji} 
		\frac{q^{i}\,\sumI_{i}^{(\ell-(j+1))}}{\product[\ell+i,\ell+i+j]}   +\sum_{i=1}^{\ell-(j+1)}(-1)^{i}q^{ji}\frac{q^{i}\,\sumI_{i-1}^{(\ell-(j+1))}}{\product[\ell+i,\ell+i+j]}
		\\
		&=\sum_{i=0}^{\ell-1-(j+1)}(-1)^{i}q^{ji}\left[
		\frac{q^{i}\,\sumI_{i}^{(\ell-(j+1))}}{\product[\ell+i,\ell+i+j]}
		-
		\frac{q^{j+i+1}\,\sumI_{i}^{(\ell-(j+1))}}{\product[\ell+i+1,\ell+i+1+j]}
		\right]\\
		&=\sum_{i=0}^{\ell-1-(j+1)}
		(-1)^{i}
		\frac
		{q^{(j+1)i}\sumI_{i}^{(\ell-(j+1))}}
		{\product[\ell+i+1,\ell+i+j]}
		\left[
		\frac{1}{1-q^{\ell+i}}
		-
		\frac{q^{j+1}}{1-q^{\ell+i+1+j}}
		\right]
		\\
		&=\sum_{i=0}^{\ell-1-(j+1)}
		(-1)^{i}
		\frac
		{q^{(j+1)i}\sumI_{i}^{(\ell-(j+1))}}
		{\product[\ell+i+1,\ell+i+j]}
		\left[
		\frac{1-q^{\ell+i+1+j}-q^{j+1}+  q^{j+1+\ell+i}}{(1-q^{\ell+i})(1-q^{\ell+i+1+j})}
		\right]
		\\
		&=(1-q^{j+1})\sum_{i=0}^{\ell-1-(j+1)}(-1)^{i}q^{(j+1)i}\frac{\sumI_{i}^{(\ell-(j+1))}}{\product[\ell+i,\ell+i+j+1]}.
	\end{align*}

	Applying \eqref{Eq:Qk:InductionStep} $\ell-1$ times  gives
	\begin{align*}
		\sum_{i=0}^{\ell-1}(-1)^{i}\frac{\sumI_{i}^{(\ell)}}{\product[\ell+i,\ell+i]} &=
		(1-q)(1-q^2)\ldots(1-q^{\ell-1})\sum_{i=0}^{0}(-1)^{i}q^{(\ell-1)i}\frac{\sumI_{i}^{(1)}}{\product[\ell+i,\ell+i+\ell-1]}\\
		&= 		\frac{\product[1,\ell-1]}{\product[\ell,2\ell-1]}  ,
	\end{align*} 	
	which is the statement of the lemma.
\end{proof}

\begin{proof}[Proof of Lemma~\ref{Lem:AuxiliarySum}]
	In the proof we use $\sumI^{(j)}_{i}$ introduced in \eqref{Eq:Def:SsI} 
	
	\begin{align*}
		\sum_{j=0}^{k-\ell+1}\left(q^{\ell}\right)^j\prod_{r=1}^{\ell-1}(1-q^{j+r})
		&=\sum_{j=0}^{k-\ell+1}\left(q^{\ell}\right)^j\sum_{I\subseteq [\ell-1]}(-1)^{|I|}\prod_{r\in I}q^{j+r}\\
		&=\sum_{j=0}^{k-\ell+1}\left(q^{\ell}\right)^j
		\sum_{i=0}^{\ell-1}(-1)^i(q^{j})^i\sum_{I\subseteq [\ell-1], |I|=i}q^{s(I)}\\
		&=\sum_{i=0}^{\ell-1}(-1)^i\sumI^{(\ell)}_{i}\sum_{j=0}^{k-\ell+1}\left(q^{\ell+i}\right)^j\\
		&=\sum_{i=0}^{\ell-1}(-1)^i\sumI^{(\ell)}_{i}\frac{1-(q^{\ell+i})^{k-\ell+2}}{1-q^{\ell+i}}.
	\end{align*}
	
	Observe that, for $q=q(n)\in (0,1)$, $k\ge \ell -\log_q n$,  $0\le i\le \ell-1$, and $\ell=O(1)$ we have
	\[
	(q^{\ell+i})^{k-\ell+2}\le  q^{\ell(k-\ell)}  \leq    n^{-\ell}.
	\]
	Therefore
	\[
	\sum_{j=0}^{k-\ell+1}\left(q^{\ell}\right)^j\prod_{r=1}^{\ell-1}(1-q^{j+r})
	= 
	(1+ O(n^{-\ell}) )\sum_{i=0}^{\ell-1}(-1)^i\frac{\sumI^{(\ell)}_{i}}{1-q^{\ell+i}}, 
	\]and so the result follows by applying Theorem~\ref{Lem:Qk:SumIProp}.
\end{proof}

\subsection{Proof of Theorem~\ref{Thm:Qk:value}}

\begin{proof}[Proof of Theorem~\ref{Thm:Qk:value}]
	
	Events $\cA[0,i]$, $i=0,1,2,\ldots$, are mutually exclusive and $\cB[0]=\bigcup_{i\in \NN_0}\cA[0,i]$. Therefore
	\begin{multline}\label{Eq:Q0}
		\bB_{0}=\Pra{\cB[0]}=\sum_{i=0}^{\infty}\Pra{\cA[0,i]}=\\=\sum_{i=0}^{\infty}\left(\prod_{j=1}^{i}\Pra{Z_j>\ell}\right)\Pra{\cA[0,0]}
		=\sum_{i=0}^{\infty}(q^\ell)^i\Pra{\cA[0,0]}=\frac{\productO}{1-q^\ell},
	\end{multline}
	where $i$ represents the number of times that $Z_j>\ell$ before in the Mallows infinite process one chooses the first number in $[\ell]$. By convention we set $\prod_{j=1}^0a_j=1$.
	
	Now let us assume that $k\ge 1$.
	Before we study a general formula for $\bB_{k}$ we start with the value of $\Pra{\cA[k,0]}$. By the properties of the geometric distribution $\operatorname{Geom}(1-q)$ we have 
	\begin{equation}\label{Eq:Qk:Ak0}
		\begin{split}
			\Pra{\cA[k,0]}
			&=\sum_{\sigma\in \cS}\prod_{j=1}^{\ell}\Pra{Z_j=k+\sigma(j)-|\sigma([j-1])\cap [\sigma(j)-1]|}\\
			&=\sum_{\sigma\in \cS}\prod_{j=1}^{\ell}q^k\Pra{Z_j=\sigma(j)-|\sigma([j-1])\cap [\sigma(j)-1]|}\\
			&=q^{k\ell}\productO.
		\end{split}
	\end{equation}
	
	Now, we determine the recursive formula for $\bB_{k}$. As before, $i$ counts the number of times the geometric random variable $Z_j$, $j=1,2,\ldots$, in the infinite Mallows process, takes values greater than $k+\ell$, before it takes a value at most $k+\ell$.
	\begin{align}\nonumber
		\bB_{k}
		&=\sum_{i=0}^{\infty}\left(\prod_{j=1}^{i}\Pra{Z_j>k+\ell}\right)\left(\Pra{\cA[k,0]}+\Pra{Z_{i+1}\le k}\bB_{k-1}\right)\\
		\nonumber
		&=\sum_{i=0}^{\infty}(q^{\ell+k})^i\left(\Pra{\cA[k,0]}+(1-q^k)\bB_{k-1,l}\right)\\
		\label{Eq:Qk:recurrence}
		&=\frac{q^{k\ell}\productO}{1-q^{\ell+k}}+\frac{(1-q^k)}{1-q^{k+\ell}}\bB_{k-1},
	\end{align} 
	where we used \eqref{Eq:Qk:Ak0} in the last line.
	By applying \eqref{Eq:Qk:recurrence} one can easily show by induction that, for any $k\in \bN_0$,
	\begin{equation}\label{Eq:Qk:Sum0}
		\bB_{k}=\sum_{i=0}^k q^{i\ell}\frac{\product[i+1,k]\productO}{\product[\ell+i,\ell+k]}.
	\end{equation}
	
	Now we split the sum \eqref{Eq:Qk:Sum0} into two. 
	\begin{equation}\label{Eq:Qk:Sum01}
		\bB_{k}=\sum_{i=0}^{k-\ell+1} q^{i\ell}\frac{\product[i+1,k]\productO}{\product[\ell+i,\ell+k]}+\sum_{i=k-\ell+2}^{k} q^{i\ell}\frac{\product[i+1,k]\productO}{\product[\ell+i,\ell+k]}.
	\end{equation}

	The first sum in \eqref{Eq:Qk:Sum01} contains the terms with $i\le k-\ell+1$. For $i\le k-\ell+1$ we have
	\begin{equation}\label{Eq:Qk:cancelation}
		\frac{\product[i+1,k]}{\product[\ell+i,\ell+k]}
		=
		\frac{\product[i+1,i+\ell-1]\product[i+\ell,k]}{\product[i+\ell,k]\product[k+1,\ell+k]}
		=\frac{\product[i+1,i+\ell-1]}{\product[k+1,\ell+k]}.
	\end{equation}
	Moreover, we consider $k\ge \ell -\log_q n$. For them
	\[
	q^k\le q^{\ell-\log_q n}\le n^{-1}.\]
	Therefore
	\begin{equation}\label{Eq:Qk:bottomtems1}  
		1\ge\product[k+1,\ell+k] = \prod_{j=1}^{\ell}(1-q^{k+i}) \ge (1-q^{k})^\ell \ge (1-n^{-1})^\ell\ge 1-\ell n^{-1}=1+O(n^{-1}).
	\end{equation}
	Similarly, for the second sum in \eqref{Eq:Qk:Sum01} we have $i=k-\ell+2,\ldots,k$.  Using \eqref{Eq:Qk:bottomtems1} 
	\begin{equation}\label{Eq:Qk:bottomtems2}  
		1\ge \product[\ell+i,\ell+k]\ge \product[k+1,\ell+k] = 1+O(n^{-1}).
	\end{equation}
	In addition, for $i=k-\ell+2,\ldots,k$, since $k\ge \ell -\log_q n$,
	\begin{equation}\label{Eq:Qk:topterms}
		q^{i\ell}\le q^{(k-\ell)\ell}\le n^{-\ell}=O(n^{-\ell}). 
	\end{equation}
	And obviously $\product[\ell+i,\ell+k]\le 1$. Therefore
	we apply \eqref{Eq:Qk:cancelation} and \eqref{Eq:Qk:bottomtems1} for the first sum and \eqref{Eq:Qk:bottomtems2} and \eqref{Eq:Qk:topterms} for the second sum in \eqref{Eq:Qk:Sum01} to obtain
	\begin{equation}\label{Eq:Qk:Sum1}
		\begin{split}
			\bB_{k}
			&=\productO\sum_{i=0}^{k-\ell+1}\left(q^{\ell}\right)^i\frac{\product[i+1,i+\ell-1]}{\product[k+1,k+\ell]}+\sum_{i=k-\ell+2}^{k}q^{i\ell}\frac{\product[i+1,k]\productO}{\product[\ell+i,\ell+k]}\\
			&=(1+O(n^{-1}))\productO\sum_{i=0}^{k-\ell+1}\left(q^{\ell}\right)^i\product[i+1,i+\ell-1]+O(n^{-\ell}\productO)\\
			&=(1+O(n^{-1}))\productO\left(\sum_{j=0}^{k-\ell+1}\left(q^{\ell}\right)^j\prod_{r=1}^{\ell-1}(1-q^{j+r})\right)+O(n^{-\ell}\productO).
		\end{split}
	\end{equation}
	The first part of Theorem~\ref{Thm:Qk:value} follows immediately from Lemma~\ref{Lem:AuxiliarySum} applied to the above equation.
	
	Now we assume that $1-q=o(1)$. Then for any $j=O(1)$ we have
	\[
	1-q^j=1-[1-(1-q)]^j=j(1-q)(1+O(1-q)).
	\]
	Therefore
	\begin{align}\nonumber
		\bB_{k}
		&=(1+O(n^{-1}))\frac{\productO}{1-q^\ell}\prod_{i=1}^{\ell-1}\frac{(1-q^i)}{(1-q^{\ell+i})}+O(\productO n^{-\ell})\\	\nonumber		&=(1+O(n^{-1}+O(1-q)))\frac{\productO}{\ell(1-q)}\frac{(\ell-1)!\ell!}{(2\ell-1)!}+O(\productO n^{-\ell})\\
		\label{Eq:q:tendsto1}			&=(1+O(n^{-1}+O(1-q)))\frac{\productO}{1-q}\frac{[(\ell-1)!]^2}{(2\ell-1)!}.
	\end{align}
\end{proof}

\subsection{Auxiliary lemmas}

Now we prove two lemmas that will be used in some of the remaining proofs of the article. Since some of them relate to Theorem~\ref{Thm:ExpectedValueProcess} and some proofs use parts of the reasoning utilised before in this section, it is convenient to present them here.

\begin{lemma}\label{Lem:ProductO}
	Let $\ell\ge 2$ and $\emptyset\neq\cS\subseteq \bS_\ell$. For any $q\in (0,1)$ we have 
	\[	q^{\ell(\ell-1)/2}(1-q)^\ell\le \productO\le \prod_{i=1}^{\ell}(1-q^i)=
	\product(\bS_\ell).\]
	Moreover for $q=q(n)$ such that $1-q=o(1)$ we have
	\[
	\productO = (1+o(1)) |\cS|(1-q)^{\ell}
	\quad\text{and}\quad	\product(\bS_\ell)=(1+o(1)) \ell! (1-q)^{\ell}.
	\]
\end{lemma}

\begin{proof}[Proof of Lemma~\ref{Lem:ProductO}]
	For any $\sigma\in \bS_\ell$ there is a unique sequence $(i_{\sigma,s})_{s\in [\ell]}$, $i_{\sigma,s}\in [\ell]$ such that $Z_s=i_{\sigma,s}$, $s\in [\ell]$, generates a $(\{\sigma\},0,0)$--arrangement in $\barPi$. Moreover, for all $s\in [\ell]$, $1\le i_{\sigma,s}\le \ell+1-s$. Therefore, for $\sigma\in \bS_\ell$,
	\[
	\product(\{\sigma\})=\prod_{s=1}^{\ell}\left(q^{i_{\sigma,s}-1}(1-q)\right)
	=q^{(\sum_{s=1}^{\ell}i_{\sigma,s})-\ell}(1-q)^\ell\ge q^{(\sum_{s=1}^{\ell}s) - \ell}(1-q)^\ell\ge q^{\ell(\ell-1)/2}(1-q)^\ell.
	\]	
	
	We get the result, as by definition of $\productO$, for any $\sigma\in \cS\subseteq \bS_\ell$, and \eqref{Eq:Def:LambdaEll}
	\[
	\product(\{\sigma\})\le \productO\le \product(\bS_\ell).
	\] 
	The second part of the lemma follows, since for $1-q=o(1)$ and any $j=O(1)$ we have
	$
	q^{j}=(1+o(1)).
	$ 
	Moreover $|\bS_\ell|=O(1)$. Thus the definition of $\productO$ implies that
	\[
	\productO=\sum_{\sigma\in \cS}\product(\{\sigma\})=\sum_{\sigma\in \cS}q^{(\sum_{s=1}^{\ell}i_{\sigma,s})-\ell}(1-q)^\ell=(1+o(1))|\cS| (1-q)^\ell.
	\]
\end{proof}

\begin{lemma}\label{Lem:QkLEQ0} Let $q \in (0,1)$, $\ell\ge 2$, and $\emptyset\neq\cS\subseteq \bS_\ell$. Then, 
	\[\forall_{k\ge 0}\ \bB_{k}=\Pra{\cB[k]}\le \bB_{0} = \frac{\productO}{1-q^{\ell}}\le \prod_{i=1}^{\ell-1}(1-q^i).\]
	Moreover for $\ell\ge 3$
	\[\forall_{\substack{k_1,k_2\ge 0\\ k_2-k_1\ge \ell}}\ \Pra{\cB[k_1]\cap \cB[k_2]}\le \bB_{0}^2 =  \left(\frac{\productO}{1-q^{\ell}}\right)^2,\]
	and for $q=q(n)$ such that $1-q=o(1)$
	\begin{align*}
		&\forall_{k\ge 0}&\ \bB_{k}=\Pra{\cB[k]}&\le (\ell-1)!(1-q)^{\ell-1}, &&\text{ for }\ell\ge 2,\quad \text{and}\\
		&\forall_{\substack{k_1,k_2\ge 0\\ k_2-k_1\ge \ell}}&\Pra{\cB[k_1]\cap \cB[k_2]}&\le\left((\ell-1)! (1-q)^{\ell-1}\right)^2, &&\text{ for }\ell\ge 3.
	\end{align*}
\end{lemma}

\begin{proof}[Proof of Lemma~\ref{Lem:QkLEQ0}]
	We prove the first inequality by induction. Set $k\ge 1$ and assume that $\bB_{k-1}\le \bB_{0}$. Let us recall that by \eqref{Eq:Q0} $\bB_{0}=\productO/(1-q^{\ell})$. In addition we use \eqref{Eq:Qk:recurrence} and the formula $1-x^i=(1-x)(1+x+\ldots+x^{i-1})$ to obtain 
	\begin{equation*}
		\begin{split}
			\bB_{k}
			&=\frac{q^{k\ell}\productO}{1-q^{\ell+k}}+\frac{(1-q^k)}{1-q^{k+\ell}}\bB_{k-1}\\
			&\le \frac{q^{k\ell}(1-q^{\ell})+1-q^k}{1-q^{k+\ell}}\bB_{0}\\
			&=\frac{1+q+\ldots+q^{k-1}+q^{k\ell}+q^{k\ell+1}+\ldots+q^{k\ell+\ell-1}}{1+q+\ldots+q^{k+\ell-1}}\bB_{0}\\
			&\le \bB_{0}.
		\end{split}
	\end{equation*} 
	To get the second equation, we apply \eqref{Eq:Q0}. The last inequality follows because, in the fraction in the second-to-last line, the number of terms in the sum in the numerator is the same as in the denominator. Moreover, these terms may be paired off in such a way that the terms in the numerator are at most equal to those in the denominator. 
	
	The proof of the second formula is similar to the proof for $\bB_k$. Let $\bB_{k_1,k_2}=\Pra{\cB[k_1]\cap \cB[k_2]}$. We use a similar reasoning to this used in the proof of \eqref{Eq:Q0}. First consider $\bB_{0,\ell}$. In the formula below the term $q^{2\ell i}$ is the probability that in the Mallows infinite process $Z_j>2\ell$ for $j=1,2,\ldots,i$, i.e. ${\mathbf \Pi}(j)>2\ell$ for $j=1,2,\ldots,i$.  In $\productO\bB_0$, the term $\productO$ is the probability that $\{i+1,\ldots,i+\ell\}$ form an $(\cS,0,i)$--arrangement in $\barPi$ and the term $\bB_0$ is the probability that $\{i+\ell+1,\ldots,i+2\ell\}$ form $(\cS,\ell,k)$--arrangement in $\barPi$, for some $k\ge i+\ell$.   In $q^{\ell^2}\productO\bB_0$, the term $q^{\ell^2}\productO$ is the probability that $\{i+1,\ldots,i+\ell\}$ form an $(\cS,\ell,i)$--arrangement in $\barPi$ and the term $\bB_0$ is the probability that $\{i+\ell+1,\ldots,i+2\ell\}$ form $(\cS,0,k)$--arrangement in $\barPi$, for some $k\ge i+\ell$. 
	\[
	\begin{split}
		\bB_{0,\ell}
		&=\sum_{i=0}^{\infty}q^{2\ell i}\left(\productO\bB_0+q^{\ell^2}\productO\bB_0\right)\\
		&=\frac{1+q^{\ell^2}}{1-q^{2\ell}}\productO\bB_0=\frac{1+q^{\ell^2}}{(1+q^\ell)(1-q^{\ell})}(1-q^{\ell})\bB_0^2\le \bB_0^2.
	\end{split}
	\]
	In the third equality we used $\eqref{Eq:Q0}$. 
	For the inductive step, we use recurrence formulas derived analogously to \eqref{Eq:Qk:recurrence}; cf. the above equation. We obtain
	for $k_1\ge 1$ and $k_2\ge k_1+\ell+1$
	\begin{align*}
		\bB_{k_1,k_1+\ell}&=\sum_{i=0}^{\infty}q^{i(k_1+2\ell)}\left((1-q^{k_1})\bB_{k_1-1,k_1-1+\ell}+q^{k_1\ell}\left(\productO\bB_{k_1}+q^{\ell^2}\productO\bB_{k_1}\right)\right)\\
		&=\frac{1}{1-q^{k_1+2\ell}}\left((1-q^{k_1})\bB_{k_1-1,k_1-1+\ell}+q^{k_1\ell}\left(\productO\bB_{k_1}+q^{\ell^2}\productO\bB_{k_1}\right)\right)\\
		\bB_{k_1,k_2}&=\sum_{i=0}^{\infty}q^{i(k_2+2\ell)}
		\Big((1-q^{k_1})\bB_{k_1-1,k_2-1}
		+q^{k_1\ell}\productO\bB_{k_2-\ell}\\
		&\quad\hspace{4.5cm}+
		q^{k_1+\ell}(1-q^{k_2-k_1+\ell})\bB_{k_1,k_2-1}
		+q^{k_2\ell}\productO\bB_{k_1}
		\Big)\\
		&=\frac{1}{1-q^{k_2+\ell}}
		\Big((1-q^{k_1})\bB_{k_1-1,k_2-1}
		+q^{k_1\ell}\productO\bB_{k_2-\ell}\\
		&\quad\hspace{4.5cm}+
		q^{k_1+\ell}(1-q^{k_2-k_1+\ell})\bB_{k_1,k_2-1}
		+q^{k_2\ell}\productO\bB_{k_1}
		\Big).\\
	\end{align*} 
	Therefore for $k_1\ge 1$
	\begin{align*}
		\bB_{k_1,k_1+\ell}&\le \frac{1}{1-q^{k_1+2\ell}}\left((1-q^{k_1})\bB_{0}^2+q^{k_1\ell}\productO\bB_{0}+q^{k_1\ell+\ell^2}\productO\bB_{0}\right)\\
		&=\frac{(1-q^{k_1})+q^{k_1\ell}(1-q^\ell)+q^{k_1\ell+\ell^2}(1-q^\ell)}{1-q^{k_1+2\ell}}\bB_{0}^2\\
		&=\frac{(1+\ldots+q^{k_1-1})+(q^{k_1\ell}+\ldots+q^{k_1\ell+\ell-1})+(q^{k_1\ell+\ell^2}+\ldots+q^{k_1\ell+\ell^2+\ell-1})}
		{(1+\ldots+q^{k_1-1})+(q^{k_1}+\ldots+q^{k_1+\ell-1})+(q^{k_1+\ell}+\ldots+q^{k_1+\ell+\ell-1})}\bB_{0}^2\\
		&\le \bB_{0}^2.
	\end{align*}
	For $k_2\ge k_1+\ell+1$ 
	\begin{align*}
		&\bB_{k_1,k_2}\le \frac{(1-q^{k_1})+q^{k_1\ell}(1-q^\ell)+q^{k_1+\ell}(1-q^{k_2-k_1+\ell})+q^{k_2\ell}(1-q^\ell)}{1-q^{k_2+\ell}}\bB_{0}^2\\
		&=\frac{
			(1+\ldots+q^{k_1-1})
			+(q^{k_1\ell}+\ldots+q^{k_1\ell+\ell-1})
			+(q^{k_1+\ell}+\ldots+q^{k_2+2\ell-1})
			+(q^{k_2\ell}+\ldots+q^{k_2\ell+\ell-1})}
		{(1+\ldots+q^{k_1-1})
			+(q^{k_1}+\ldots+q^{k_1+\ell-1})
			+(q^{k_1+\ell}+\ldots+q^{k_2+2\ell-1})
			+(q^{k_2+2\ell}+\ldots+q^{k_2+3\ell-1})}\bB_{0}^2\\
		&\le \bB_{0}^2.
	\end{align*}
	In the last inequality we have used the fact that for $\ell\ge 3$ and $k_2\ge k_1+\ell+1\ge \ell+1$ we have $k_2\ell\ge 3k_2> k_2+2\ell$.
\end{proof}

\subsection{Proof of Theorem~\ref{Thm:ExpectedValueProcess}}

\begin{proof}[Proof of Theorem~\ref{Thm:ExpectedValueProcess}]

	Let  $\ell\ge 2$, $q\in (0,1)$, $t(1-q)/\ln n\to \infty$ and $t\to\infty$ as $n\to\infty$, and $K=\ell-\log_q n$. Since $\ln(1-x)\le -x$, 
	we have
	\begin{equation}\label{Eq:K:Upper}
		K=\ell-\log_q n=\ell+\frac{\ln n}{-\ln [1-(1-q)]} \le \ell+\frac{\ln n}{(1-q)} = \ell+t \frac{\ln n}{t(1-q)}=o(t).
	\end{equation}
	Let $\Ind_i$ be the indicator random variable of the event~$\cB[i]$ that $(\barPi_{i+1},\ldots,\barPi_{i+\ell})$ is an $\cS$--arrangement. Then $\bE \Ind_i = \Pra{\cB[i]}=\bB_i$ and $\Arr_{\cS}(\barPi|_{[t]})=\sum_{i=0}^{t-\ell}\bE \Ind_i$. 
	By Lemma~\ref{Lem:QkLEQ0} and $\eqref{Eq:K:Upper}$ we get 
	\[
	\sum_{i=0}^{K-1}\bE \Ind_i = \sum_{i=0}^{K-1}\bB_{i}\le K\,  \bB_{0} = o\left(t\, \frac{\productO}{(1-q^\ell)}\right).
	\]
	Thus, by Theorem~\ref{Thm:Qk:value}, for any $q\in (0,1)$, $t\to\infty$, and $(1-q)t/\ln n\to \infty$  as $n\to \infty$,
	\[
	\begin{split}
		\bE \Arr_\cS(\barPi|_{[t]}) 
		&=\sum_{i=0}^{K-1}\bE\Ind_i + \sum_{i=K}^{t-\ell}\bE\Ind_i\\
		&=o\left(t\frac{\productO}{1-q^\ell}\right)+(t-\ell-K+1)(1+o(1))\frac{\productO}{1-q^{\ell}}\prod_{j=1}^{\ell-1}\frac{(1-q^j)}{(1-q^{\ell+j})}+O(tn^{-\ell})\\
		&=(1+o(1))\frac{\productO}{1-q^{\ell}}\left(\prod_{j=1}^{\ell-1}\frac{(1-q^j)}{(1-q^{\ell+j})}\right)t.
	\end{split}
	\]	
	We recall here that $(1-q^j)/(1-q^{\ell+j})=\Theta(1)$, which we used in the last equality.
	
	In particular, for $1-q=o(1)$, using the second part of Theorem~\ref{Thm:Qk:value} and the fact that by Lemma~\ref{Lem:ProductO} $\productO/(1-q^{\ell})=(1+o(1)) |\cS|(\ell-1)!(1-q)^{\ell-1}$ we get for $1-q=o(1)$
	\[
	\begin{split}
		\bE \Arr_\cS(\barPi|_{[t]}) 
		&=(1+o(1))\frac{\productO}{1-q^{\ell}}\left(\prod_{j=1}^{\ell-1}\frac{(1-q^j)}{(1-q^{\ell+j})}\right)t\\
		&=(1+o(1))\frac{|\cS|[(\ell-1)!]^2}{(2\ell-1)!}(1-q)^{\ell-1}t.
	\end{split}
	\]	
\end{proof}

\section{Proof of Theorem~\ref{Thm:Main:Expected} }

Let $\ell\ge 2$ be a constant integer, $\emptyset\neq\cS\subseteq \bS_\ell$ and $\barPi\sim \Mallows(\NN,q)$. 
For the Mallows infinite sequence $\barPi\sim \Mallows(\NN,q)$,  let us define event
\begin{equation}\label{Eq:E:Def}
	\cE=\{\barPi([n-8\Delta])\subseq \barPi|_{[t-4\Delta]}\subseq \barPi([n])\}.
\end{equation}
Lemma~\ref{Lem:Inclusion0} and Lemma~\ref{Lem:Inclusion2}, with $x=4$, imply
\begin{equation}\label{Eq:E:Prob}
	\Pra{\cE'}\le 3n^{-3}.
\end{equation}
If $\cE$ occurs, then by Lemma~\ref{Lem:Inclusion:Implication}
\begin{equation}\label{Eq:Arr:Coup}
	\Arr_{\cS}(\barPi|_{[n-4\Delta]})\le \Arr_{\cS}(\barPi([n]))\le \Arr_{\cS}(\barPi|_{[n-4\Delta]})+\Arr_{\cS}(\barbarPi([n-12\Delta-\ell+1,n])).
\end{equation}
Corollary~\ref{Cor:varPi:to:barPi}, \eqref{Eq:Arr:Coup}, and \eqref{Eq:E:Prob} imply
\[
\begin{split}
	\bE(\Arr_{\cS}(\varPi_n))
	=
	\bE(\Arr_{\cS}(\barPi([n])))
	&\ge
	\bE(\Arr_{\cS}(\barPi([n])))\Ind_{\cE})
	\\
	&\ge 
	\bE(\Arr_{\cS}(\barPi|_{[n-4\Delta]})(1-\Ind_{\cE'}))\\
	&\ge \bE(\Arr_{\cS}(\barPi|_{[n-4\Delta]}))-n\bE \Ind_{\cE'}\\
	&\ge \bE\Arr_{\cS}(\barPi|_{[n-4\Delta]})-3n^{-2}.
\end{split}	
\]
In addition, by Lemma~\ref{Lem:CoupligProcesses} $\barbarPi([n-12\Delta-\ell,n])\sim\Mallows (12\Delta+\ell,q)$ and by Theorem~\ref{Thm:Pinsky} with Remark~\ref{Rem:Thm:Pinsky} (since $\bE\Arr_{\cS}(\varPi_t)\le \bE\Arr_{\bS_\ell}(\varPi_t)$) 
\[\bE \barbarPi([n-12\Delta-\ell,n]) \le (1+o(1)) (\ell-1)!(1-q)^{\ell-1}(12\Delta+\ell)=O\left(\Delta(1-q)^{\ell-1}\right). \]
Therefore Corollary~\ref{Cor:varPi:to:barPi}, \eqref{Eq:Arr:Coup},  the above equation, and \eqref{Eq:E:Prob} imply 
\[
\begin{split}
	\bE(\Arr_{\cS}(\varPi_n)
	&=\bE(\Arr_{\cS}(\barPi([n]))\\
	&\le 
	\bE(\Arr_{\cS}(\barPi([n]))\Ind_{\cE})+\bE(\Arr_{\cS}(\barPi([n]))\Ind_{\cE'})\\
	&\le \bE(\Arr_{\cS}(\barPi|_{[n-4\Delta]}))+\bE(\Arr_{\cS}(\barbarPi([n-12\Delta-\ell,n])))+n\bE \Ind_{\cE'}\\
	&\le \bE(\Arr_{\cS}(\barPi|_{[n-4\Delta]}))+O\left(\Delta(1-q)^{\ell-1}\right)+3n^{-2}.
\end{split}	
\]
Thus by Theorem~\ref{Thm:ExpectedValueProcess}
\begin{align*}
	\bE(\Arr_{\cS}(\varPi_n))&=\bE(\Arr_{\cS}(\barPi|_{[n-4\Delta]}))+O\left(\Delta(1-q)^{\ell-1}+n^{-2}\right)\\
	&=
	(1+o(1)) \frac{\productO}{1-q^\ell}\frac{\prod_{i=1}^{\ell-1}(1-q^i)}{\prod_{i=1}^{\ell-1}(1-q^{\ell+i})}n	+O\left(\Delta(1-q)^{\ell-1}+n^{-2}\right)
\end{align*}
Now let us note that since $n(1-q)/\ln n \to \infty$ as $n\to \infty$, we have
\[\Delta=\Bigg\lceil \frac{\ln n}{1-q}\Bigg\rceil = o(n).\]
Moreover by Lemma~\ref{Lem:ProductO}, for $q=q(n)$ bounded away from $0$ we have
\[\productO=\Theta((1-q)^{\ell}).\]
As shown in \eqref{Eq:q:tendsto1} , for any $q=q(n)\in (0,1)$
\[\frac{\prod_{i=1}^{\ell-1}(1-q^i)}{(1-q^\ell)\prod_{i=1}^{\ell-1}(1-q^{\ell+i})}=\Theta\left(\frac{1}{1-q}\right).\]
This implies the first part of the statement of Theorem~\ref{Thm:Main:Expected}. The second one follows similarly using the second part of Theorem~\ref{Thm:ExpectedValueProcess}.

\section{Proof of Theorem~\ref{Thm:Main:Poisson}}

In the proof of Theorem~\ref{Thm:Main:Poisson} we use  Stein--Chen method \cite{Chen1975, SteinBook}.
We follow the approach of Janson \cite{Janson1994} and Barbour, Holst, and Janson  \cite{BarbourHolstJansonBook}.
The following lemma will allow us to concentrate on $\Arr_{\cS}(\barPi|_{[6\Delta+1,n-6\Delta]})$, $\barPi\sim\Mallows(\NN,q)$, instead of $\Arr_{\cS}(\barPi([n]))$, $\barPi_n\sim\Mallows(\bN,q)$. Recall that by Lemma~\ref{Lem:CoupligProcesses} we have $\barPi([n])\sim \Mallows(n,q)$. 

\begin{lemma}\label{Lem:InclusionDTV}
	Let $\ell\ge 2$ be a constant, $q=q(n)\in (0,1)$ be such that $n(1-q)\ge 1$, $\emptyset\neq\cS\subseteq \bS_\ell$,  $\barPi\sim \Mallows(\NN,q)$, and $\varPi_{12\Delta+\ell}\sim \Mallow(12\Delta+\ell,q)$. Then 
	\[
	\dtv\left(\Arr_{\cS}(\barPi|_{[6\Delta+1,n-6\Delta]}),\Arr_{\cS}(\barPi[n])\right)
	\le 5n^{-5}+2\bE(\Arr_{\cS}(\varPi_{12\Delta+\ell})).
	\]
\end{lemma}
The proof of the lemma is postponed to a later part of the section.
Now concentrate on
$\Arr_{\cS}(\barPi|_{[6\Delta+1,n-6\Delta]})$. 
Let $\Gamma=[6\Delta,n-6\Delta-\ell]$. Recall that $\cB[k]$ is the event that $(\barPi_{k+1},\ldots,\barPi_{k+\ell})$ is an $\cS$--arrangement in $\barPi\sim\Mallows(\NN,q)$ and for each $k\in \Gamma$
\begin{equation}\label{Eq:Def:Indk}
	\Ind_k=\Ind_{\cB[k]},\quad \bB_k=\Pra{\Ind_k=1}=\bE \Ind_k. 
\end{equation}
are the indicator random variable of $\cB[k]$ and its expected value.  Moreover we define
\[\bW=\bW(n)=\Arr_\cS(\barPi|_{[6\Delta+1,n-6\Delta]})=\sum_{k\in\Gamma}\Ind_k\]
and
\[\lambda = \lambda(n) = \bE \bW = \sum_{k\in \Gamma} \bB_{k},\]
Note that
\begin{equation}\label{Eq:ell:Delta}
	\ell - \ln_q n = \ell + \frac{\ln n}{-\ln q} \le \ell + \frac{\ln n}{1-q} \le 6\Delta.
\end{equation}
Therefore by Theorem~\ref{Thm:Qk:value}, since $1-q=o(1)$,
\begin{equation}\label{Eq:Lambda}
	\lambda=(1+o(1))\frac{|\cS|[(\ell-1)!]^2}{(2\ell-1)!}n(1-q)^{\ell-1}=\Theta(n(1-q)^{\ell-1}).
\end{equation}
For each $k\in \Gamma$, we split the index set $\Gamma=\{k\}\cup\Gamma_{k,1}\cup\Gamma_{k,2}$, where $\Gamma_{k,1}$ contains all the indices in $\Gamma\setminus\{k\}$ that are at distance at most $12\Delta+\ell$ from $k$ and $\Gamma_{k,2}$ consists of all the remaining indices of $\Gamma\setminus\{k\}$. That is

\begin{align}
	\Gamma_{k,1}&=[\max\{k-12\Delta-\ell ,6\Delta\},\min\{k+12\Delta+\ell,n-6\Delta\}\}]\setminus \{k\}\nonumber \\
	\Gamma_{k,2}&=\Gamma\setminus [\max\{k-12\Delta-\ell,6\Delta\},\min\{k+12\Delta+\ell,n-6\Delta\}\}].\label{Eq:Def:Gammak}
\end{align}
Note that for each $k\in \Gamma$
\begin{equation*}
	|\Gamma_{k,1}|\le  24\Delta + 2\ell.
\end{equation*}
Now, given $k\in \Gamma$, we split $\bW=\Ind_k+\bW_1+\bW_2$, where
\begin{equation}\label{Eq:Def:W12}
	\bW_i=\sum_{j\in \Gamma_i}\Ind_j,\quad i\in \{1,2\}. 
\end{equation}
Now let $\bW_{k,1}$ and $\bW_{k,2}$ to be any random variables with the joint conditional distributions as $\bW_{1}$ and $\bW_{2}$, respectively, under condition $\Ind_{k}=1$, i.e.
\begin{equation}\label{Eq:Def:Wk12}
	{\cal L}(\bW_{k,1},\bW_{k,2})={\cal L}(\bW_{1},\bW_{2}|\Ind_{k}=1).
\end{equation}
Moreover let $\bW_k=\bW_{k,1}+\bW_{k,2}$.
Using Theorem~1 \cite{Janson1994} (see also \cite{BarbourHolstJansonBook}) and some further reasoning from \cite{Janson1994,BarbourHolstJansonBook}  we obtain 
\begin{align}\nonumber
	\dtv&\left(\Arr_{\cS}(\barPi|_{[6\Delta+1,n-6\Delta]}),\operatorname{Po}(\lambda)\right)=\\
	&=\dtv(\bW,\operatorname{Po}(\lambda))\\ \nonumber
	&\le (1\wedge \lambda^{-1})\sum_{k\in \Gamma}\bB_k\bE|\bW-\bW_k|\\
	&\le (1\wedge \lambda^{-1})
	\left(
	\sum_{k\in \Gamma}\bB_k^2
	+\sum_{k\in \Gamma}\bB_k\bE|\bW_1-\bW_{k,1}|
	+\sum_{k\in \Gamma}\bB_k\bE|\bW_2-\bW_{k,2}|
	\right)\label{Eq:Poisson:Dtv:Suma}.
\end{align}
Theorem~\ref{Thm:Main:Poisson} will follow by Lemma~\ref{Lem:InclusionDTV}, \eqref{Eq:Poisson:Dtv:Suma}, and the following lemma. 
\begin{lemma}\label{Lem:dtv:Coupling}
	Let $\Ind_k$ and $\bB_k$ be as in \eqref{Eq:Def:Indk} and $\bW_1$, $\bW_2$, $\bW_{k,1}$, $\bW_{k,2}$ be as in \eqref{Eq:Def:W12}  and  \eqref{Eq:Def:Wk12}, $1-q=o(1)$, and $ n^{-5/(\ell-1)}=o(1-q)$. Then
	\begin{align}
		\label{Eq:dtv:Coupling:1}	\sum_{k\in \Gamma}\bB_k^2&=O(n(1-q)^{2\ell-2});\\
		\label{Eq:dtv:Coupling:2}	\sum_{k\in \Gamma}\bB_k\bE|\bW_1-\bW_{k,1}|&=O\left(n\ln n(1-q)^{2\ell-3}+n(1-q)^{\ell}\right);\\
		\intertext{ and there exists a coupling of $\bW_2$ and $\bW_{k,2}$ for which}
		\label{Eq:dtv:Coupling:3}	\sum_{k\in \Gamma}\bB_k\bE|\bW_2-\bW_{k,2}|&=O(n^{-3}(1-q)^{-2(\ell-1)}).
	\end{align}	
\end{lemma}

\medskip

Now we present the proofs of Lemma~\ref{Lem:InclusionDTV}, Lemma~\ref{Lem:dtv:Coupling}, and Theorem~\ref{Thm:Main:Poisson}.

\begin{proof}[Proof of Lemma~\ref{Lem:InclusionDTV}] 
	We use the following fact, see for example Fact 4 in \cite{FillSchSC2000}.
	
	\begin{fact}\label{Fact:dtv}
		Let $X$ and $X'$ be two random variables. If there exists a probability space on which
		random variables $Y$ and $Y'$ are both defined and have probability distribution as $X$ and $X'$,
		respectively, then
		\[\dtv(X,X')\le \Pra{Y\neq Y'}.
		\]
	\end{fact}
	We will use this fact for the coupling $(\Arr_{\cS}(\barPi|_{[6\Delta+1,n-6\Delta]}),\Arr_{\cS}(\barPi([n]))$ on the probability space of $\barPi\sim \Mallows(\NN,q)$.
	Define event
	\begin{equation*}
		{\cal F}=\{\barPi|_{[6\Delta+\ell]}\subseq\barPi([12\Delta+\ell])\subseq\barPi([n-12\Delta])\subseq \barPi|_{[n-6\Delta]}\subseq \barPi([n])\}.
	\end{equation*}
	Lemma~\ref{Lem:Inclusion0} and Lemma~\ref{Lem:Inclusion2}, with $x=6$, imply
	\begin{equation*}
		\Pra{{\cal F}'}\le 5n^{-5}.
	\end{equation*}
	Note that, 
	\begin{align*}
		\Arr_{\cS}(\barPi|_{[6\Delta+1,n-6\Delta]}) &\le \Arr_{\cS}(\barPi|_{[n-6\Delta]})\\
		\Arr_{\cS}(\barPi|_{[n-6\Delta]})&\le \Arr_{\cS}(\barPi|_{[6\Delta+\ell]}) + \Arr_{\cS}(\barPi|_{[6\Delta+1,n-6\Delta]})
	\end{align*}
	Moreover by Lemma~\ref{Lem:Inclusion:Implication}
	\begin{align*}
		\barPi|_{[n-6\Delta]}&\subseq \barPi([n])\Rightarrow\\
		&\quad\quad\Rightarrow \Arr_{\cS}(\barPi|_{[n-6\Delta]})\le \Arr_{\cS}(\barPi([n]))\\ 
		\barPi([n-12\Delta])&\subseq \barPi|_{[n-6\Delta]}\subseq \barPi([n])\Rightarrow\\
		&\quad\quad\Rightarrow\Arr_{\cS}(\barPi([n]))\le \Arr_{\cS}(\barPi|_{[n-6\Delta]})+\Arr_{\cS}(\barbarPi([n-12\Delta-\ell+1,n]) \\ 
		\barPi|_{[6\Delta+\ell]}&\subseq\barPi([12\Delta+\ell])\Rightarrow\\
		&\quad\quad\Rightarrow \Arr_{\cS}(\barPi|_{[6\Delta+\ell]})\le \Arr_{\cS}(\barPi([12\Delta+\ell]) 
	\end{align*}
	Combining the above inequalities, we obtain that if ${\cal F}$ occurs we have
	\begin{align}\nonumber
		\Arr_{\cS}(\barPi|_{[6\Delta+1,n-6\Delta]})\le\phantom{ \Arr_{\cS}(\barPi([n]))}&\\ 
		\le  \Arr_{\cS}(\barPi|_{[n-6\Delta]})
		\nonumber
		\le \Arr_{\cS}(\barPi([n]))
		&\le \Arr_{\cS}(\barPi|_{[n-6\Delta]})\\
		\nonumber &\quad\quad +\Arr_{\cS}(\barbarPi([n-12\Delta-\ell+1,n])
		\\ \nonumber
		&\le \Arr_{\cS}(\barPi|_{[6\Delta+\ell]}) + \Arr_{\cS}(\barPi|_{[6\Delta+1,n-6\Delta]})
		\\ \nonumber &\quad\quad + \Arr_{\cS}(\barbarPi([n-12\Delta-\ell+1,n])
		\\ \nonumber
		&\le 
		\Arr_{\cS}(\barbarPi([12\Delta+\ell])+\Arr_{\cS}(\barPi|_{[6\Delta+1,n-6\Delta]})\\
		&\quad\quad+\Arr_{\cS}(\barbarPi([n-12\Delta-\ell+1,n]).
		\label{Eq:InclusionDtv:inequalities}
	\end{align}
	Thus, on the probability space of $\barPi\sim \Mallows(\NN,q)$, ${\cal F}$ implies 
	\begin{multline}
		0
		\le 
		\Arr_{\cS}(\barPi([n]))-\Arr_{\cS}(\barPi|_{[6\Delta+1,n-6\Delta]})\\
		\le
		\Arr_{\cS}(\barbarPi([12\Delta+\ell])+\Arr_{\cS}(\barbarPi([n-12\Delta-\ell+1,n]).
	\end{multline}
	Therefore using Fact~\ref{Fact:dtv}, Lemma~\ref{Lem:CoupligProcesses}, and Markov's inequality we obtain
	
	\[
	\begin{split}
		\dtv&\left(\Arr_{\cS}(\barPi|_{[6\Delta+1,n-6\Delta]}),\Arr_{\cS}(\barPi([n]))\right)
		\le 
		\Pra{\Arr_{\cS}(\barPi|_{[6\Delta+1,n-6\Delta]})\neq\Arr_{\cS}(\barPi([n]))}\\
		&\le 
		\Pra{{\cal F}'}+\Pra{\Arr_{\cS}(\barbarPi([12\Delta+\ell])\ge 1}+\Pra{\Arr_{\cS}(\barbarPi([n-12\Delta-\ell+1,n])\ge 1}\\
		&\le 5n^{-3}+2\bE(\Arr_{\cS}(\barpi_{12\Delta+\ell})).
	\end{split}
	\]
\end{proof}

\begin{proof}[Proof of Lemma~\ref{Lem:dtv:Coupling}] We prove each inequality separately. Recall that by \eqref{Eq:ell:Delta},  $\ell - \ln_q n < 6\Delta\le k$ for any $k\in \Gamma$.  
	\\
	\noindent{\it Proof of \eqref{Eq:dtv:Coupling:1}. }
	By Lemma~\ref{Lem:QkLEQ0} we have
	\begin{equation*}
		\sum_{k\in \Gamma}\bB_k^2\le n [(\ell-1)!(1-q)^{\ell-1}]^2 = O(n(1-q)^{2\ell-2}).
	\end{equation*}

	\noindent{\it Proof of \eqref{Eq:dtv:Coupling:2}. }
	Reasoning as in \cite{Janson1994,BarbourHolstJansonBook},  we get for the second sum
	\begin{align}\nonumber
		\sum_{k\in \Gamma}\bB_k\bE|\bW_1-\bW_{k,1}|
		&\le \sum_{k\in \Gamma}\sum_{j\in \Gamma_{1,k}}\bB_k\bB_j+\bE\Ind_k\Ind_j\\
		&\le \sum_{k\in \Gamma}\sum_{j\in \Gamma_{1,k}}\bB_k\bB_j
		+\sum_{k\in \Gamma}\sum_{\substack{j\in \Gamma_{1,k}\\ |k-j|\ge \ell}}\bE\Ind_k\Ind_j
		+\sum_{k\in \Gamma}\sum_{\substack{j\in \Gamma_{1,k}\\ 1\le |k-j| \le  \ell-1}}\bE\Ind_k\Ind_j\label{Eq:Poisson:Gamma1}.
	\end{align}
	It follows from the definition that $\bE\Ind_k\Ind_j=\Pra{\cB[k]\cap \cB[j]}$. Therefore
	the first two sums in \eqref{Eq:Poisson:Gamma1} may by bounded  greedily using Lemma~\ref{Lem:QkLEQ0} to obtain
	\[
	\sum_{k\in \Gamma}\sum_{j\in \Gamma_{1,k}}\bB_k\bB_j
	+\sum_{k\in \Gamma}\sum_{\substack{j\in \Gamma_{1,k}\\ |k-j|\ge \ell}}\bE\Ind_k\Ind_j
	\le 2n(24\Delta+2\ell)[(\ell-1)!(1-q)^{\ell-1}]^2 = O\left(n\ln n(1-q)^{2\ell-3}\right).
	\]
	Now we focus on the last sum in \eqref{Eq:Poisson:Gamma1}. Let us assume that $k<j$ (for $k>j$ the reasoning is analogous) and  $i=|k-j|$, i.e. $j=k+i$. Recall that $\bE\Ind_k\Ind_{k+i}=\Pra{\cB[k]\cap \cB[k+i]}$. Denote by $\cS_i$ the family of all permutations from $\bS_{\ell+i}$ such that every $\cS_i$--arrangement in $\barPi$ consists of two $\cS$--arrangements intersecting on $k-i$ numbers in the middle part of the $\cS_i$--arrangement. Here we allow $\cS_i=\emptyset$, but in this case all calculations are valid trivially. Then $\cB[k]\cap \cB[k+i]$ is the event that there exists $s\ge 0$ such that there is an $(\cS_i,k,s)$--arrangement in $\barPi$. Therefore, again making use of Lemma~\ref{Lem:QkLEQ0}, but now for $\cS_i$--arrangements   
	\[
	\begin{split}
		\sum_{k\in \Gamma}\sum_{\substack{j\in \Gamma_{1,k}\\ 1\le |k-j| \le  \ell-1}}\bE\Ind_k\Ind_j
		&=2\sum_{k\in \Gamma}\sum_{\substack{1\le i\le \ell-1\\ k+i\in \Gamma}}\Pra{\cB[k]\cap \cB[k+i]}\\
		&\le 2 n \sum_{i=1}^{\ell-1}(\ell+i)!(1-q)^{\ell+i-1}\\
		&=O(n(1-q)^{\ell}).
	\end{split}
	\]

	\medskip
	
	\noindent{\it Proof of \eqref{Eq:dtv:Coupling:3}.} 
	Set $k\in \Gamma$. Let us first consider the case 
	$18\Delta+\ell\le k\le n-18\Delta-\ell$. In this case $\Gamma_{k,1}=[k-12\Delta-\ell,k+12\Delta+\ell]$ as defined in \eqref{Eq:Def:Gammak}. 
	Define event
	\begin{multline}\label{Eq:Ek:Def2}
		\cE_k=
		\{\barPi|_{[k-12\Delta]}\stackrel{(i1)}{\subseq} \barPi([k-6\Delta])\stackrel{(i2)}{\subseq} \barPi|_{[k]}\subseq\\
		\subseq \barPi|_{[k+\ell]}\stackrel{(i3)}{\subseq} \barPi([k+6\Delta+\ell])\stackrel{(i4)}{\subseq} \barPi|_{[k+12\Delta+\ell]}\subseq \barPi|_{[n-6\Delta]}\stackrel{(i5)}{\subseq} \barPi([n]).
		\}
	\end{multline}
	By Lemma~\ref{Lem:Inclusion0} and Lemma~\ref{Lem:Inclusion2}
	\begin{equation}
		\Pra{\cE_k}\ge 1-8n^{-5}.
	\end{equation}
	Let us consider three permutations
	\begin{align*}
		\barpi_{(k^-)}
		&=(\barpi_{(k^-),1},\ldots,\barpi_{(k^-),k-6\Delta})=
		\barbarPi([k-6\Delta])\\
		\barpi_{(k)}
		&=(\barpi_{(k),1},\ldots,\barpi_{(k),12\Delta+\ell})=
		\barbarPi([k-6\Delta+1,k+6\Delta+\ell])\\
		\barpi_{(k^+)}
		&=(\barpi_{(k^+),1},\ldots,\barpi_{(k^+),n-k-6\Delta-\ell})=
		\barbarPi([k+12\Delta+\ell+1,n])
	\end{align*}
	By Lemma~\ref{Lem:CoupligProcesses} $\barpi_{(k^-)}$, $\barpi_{(k)}$, and $\barpi_{(k^+)}$ are independent and have $\Mallows$ distributions.
	
	Let $j\in \NN_0$. Recall that
	$\Ind_j$ is the indicator random variable of the event that $(\barPi_{j+1},\ldots,\barPi_{j+\ell})$ is an $\cS$--arrangement.  
	For $j\in \Gamma_{k,2}\cup \{k\}$ define $\Ind_j'$ to be the indicator random variable of the event that the sequence
	\begin{align*}
		\nonumber	&(\barpi_{(k^-),j+1},\dots,\barpi_{(k^-),j+\ell}), &&\text{ for } j\in [6\Delta, k-12\Delta-\ell];\\	&(\barpi_{(k),j+1-(k-6\Delta)},\dots,\barpi_{(k),j+\ell-(k-6\Delta)}), &&\text{ for } j=k;\\ 
		\nonumber	&(\barpi_{(k^+),j+1-(k+6\Delta+\ell)},\dots,\barpi_{(k^+),j+\ell-(k+6\Delta+\ell)}), &&\text{ for } j\in [k+6\Delta+\ell+1,n-6\Delta-\ell];
	\end{align*}  
	forms an $\cS$--arrangement in $\barpi_{(k^-)}$, $\barpi_{(k)}$, and $\barpi_{(k^+)}$, respectively.
	
	On the other hand, if $\cE_k$ occurs then inclusion (i1) implies
	\begin{align}
		\label{Eq:Coupl:varPi:kminus}
		\barpi_{(k^-),i}&=\barPi_i,&&\text{ for all }i\in [6\Delta+1,k-12\Delta], \\
		\intertext{and (i2) and (i3) imply}
		\label{Eq:Coupl:varPi:k}
		\barpi_{(k),i-(k-6\Delta)}&=\barPi_{i},&&\text{ for all }i\in [k+1,k+\ell].\\ 
		\intertext{Moreover (i4) and (i5) imply}
		\label{Eq:Coupl:varPi:kplus}
		\barpi_{(k^+),i-(k+6\Delta+\ell)}&=\barPi_{i},&&\text{ for all }i\in [k+\ell+12\Delta+1,n-6\Delta-\ell]. 
	\end{align}
	Thus by \eqref{Eq:Coupl:varPi:kminus}, \eqref{Eq:Coupl:varPi:k} , and \eqref{Eq:Coupl:varPi:kplus}, we have that 
	\begin{equation}
		\cE_k \subseteq \{\forall_{j\in \Gamma_{k,2}\cup\{k\}}\, \Ind_j=\Ind_j'\}.
	\end{equation}
	An analogous reasoning applies to $k\in \Gamma\setminus [18\Delta+\ell, n-18\Delta-\ell]$, however with $\cE_k$ redefined, without (i1) and (i2) for $k<18\Delta+\ell$ and without (i3) and (i4) for $k>n-18\Delta-\ell$. 
	
	Overall we have in all cases 
	\begin{equation}\label{Eq:Ek:Ik:Ikprim}
		\forall_{k\in \Gamma}\	\cE_k \subseteq \{\forall_{j\in \Gamma_{k,2}\cup\{k\}}\, \Ind_j=\Ind_j'\}\quad\text{and}\quad \Pra{\cE_k}\ge 1-8n^{-5}.
	\end{equation}
	Recall that by Lemma~\ref{Lem:CoupligProcesses} 
	\begin{align*}
		\barpi_{(k^-)}=\barbarPi([k-6\Delta])&\sim\Mallows(k-6\Delta,q),\\
		\barpi_{(k)}=\barbarPi([k-6\Delta+1,k+6\Delta+\ell])&\sim\Mallows(12\Delta+\ell,q)\\
		\barpi_{(k^+)}=\barbarPi([k+6\Delta+\ell+1,n])&\sim\Mallows(n-k-\ell-6\Delta,q)\\
		&\quad\text{and these sequences are independent.}
	\end{align*}
	We define a new random variable
	\[
	\bW_{k,2}'=\sum_{j\in \Gamma_{k,2}}\Ind'_k.
	\]
	

	Instead of constructing a coupling of $\bW_{k,2}$ and $\bW_{2}$, we will use the triangle inequality \[
	\sum_{k\in \Gamma}\bB_k\bE|\bW_2-\bW_{k,2}|\le \sum_{k\in \Gamma}\bB_k\bE|\bW_2-\bW_{k,2}'| + \sum_{k\in \Gamma}\bB_k\bE|\bW_{k,2}'-\bW_{k,2}|.
	\] 
	Now we consider two couplings, the first of $\bW_{2}$ and $\bW_{k,2}'$ on the probability space of $\barPi\sim\Mallows(\NN,q)$ and the second of $\bW_{k,2}'$ and $\bW_k$. 
	
	It follows from \eqref{Eq:Ek:Ik:Ikprim} and the definitions of $\bW_2$ and $\bW_{k,2}'$, that on the probability space of $\barPi\sim \Mallows(\NN,q)$
	\begin{equation}\label{Eq:Ek:WWprim}
		\cE_k\subseteq \{\bW_2=\bW_{k,2}'\}.
	\end{equation}
	Therefore, recalling that by \eqref{Eq:Lambda} $\lambda=\Theta(n(1-q)^{\ell-1})$ 
	\[
	\begin{split}
		\sum_{k\in \Gamma}\bB_k\bE|\bW_2-\bW_{k,2}'|
		&= \sum_{k\in \Gamma}\bB_k\left(\bE|\bW_2-\bW_{k,2}'|\Ind_{\cE_k}+\bE|\bW_2-\bW_{k,2}'|\Ind_{\cE_k'}\right)\\
		&\le\sum_{k\in \Gamma}\bB_k(0+n\bE \Ind_{\cE_k'})
		\\
		&\le  \lambda\cdot n\cdot 8n^{-5} = O(n^{-3}(1-q)^{\ell-1}).
	\end{split}
	\]
	Therefore, since $1-q=o(1)$, \[\sum_{k\in \Gamma}\bB_k\bE|\bW_2-\bW_{k,2}'| = O(n^{-3}(1-q)^{-2(\ell-1)}). \]
	
	Now we prove that $\sum_{k\in \Gamma}\bB_k\bE|\bW_{k,2}'-\bW_{k,2}|= O(n^{-3}(1-q)^{-2(\ell-1)})$ as well. By Lemma~\ref{Lem:CoupligProcesses} permutations $\barpi_{(k^-)}$, $\barpi_{(k)}$, and $\barpi_{(k^+)}$  are independent. Moreover, none of the random variables $\{\Ind'_j, j\in \Gamma_{k,2}\}$  is related to $\barpi_{(k)}$. Therefore $\{\Ind'_j,j\in \Gamma_{k,2}\}$ are independent of $\Ind'_k$ and, as a consequence, $\bW'_{k,2}$ and $\Ind'_k$ are independent. Thus \[\Pra{\{\bW'_{k,2}=s\}\cap\{\Ind'_k=1\}}=\Pra{\bW'_{k,2}=s}\Pra{\Ind'_k=1}.\]
	This and the definition of $\bW_{k,2}$ ($\bW_2$ under the condition that $\Ind_k=1$) imply
	\[
	\begin{split}
		\Pra{\bW_{k,2}'=s}&=\frac{\Pra{\{\bW'_{k,2}=s\}\cap\{\Ind'_k=1\}}}{\Pra{\Ind'_k=1}},\\
		\Pra{\bW_{k,2}=s}&=\frac{\Pra{\{\bW_{2}=s\}\cap\{\Ind_k=1\}}}{\Pra{\Ind_k=1}}.
	\end{split}
	\]
	By \eqref{Eq:Ek:Ik:Ikprim} and \eqref{Eq:Ek:WWprim} we have
	\begin{multline*}
		\{\bW_{2}=s\}\cap\{\Ind_k=1\}\cap\cE_k=\{\bW'_{k,2}=s\}\cap\{\Ind'_k=1\}\cap\cE_k
		\\ \text{and}
		\quad
		\{\Ind'_k=1\}\cap \cE_k = \{\Ind_k=1\}\cap \cE_k. 
	\end{multline*}
	Thus 
	\begin{align*}
		&|\Pra{\{\bW_{2}=s\}\cap\{\Ind_k=1\}}-\Pra{\{\bW'_{k,2}=s\}\cap\{\Ind'_k=1\}}|=\\
		&=|\Pra{\{\bW_{2}=s\}\cap\{\Ind_k=1\}\cap\cE_k'}-\Pra{\{\bW'_{k,2}=s\}\cap\{\Ind'_k=1\}\cap\cE_k'}\\
		&\quad+\Pra{\{\bW_{2}=s\}\cap\{\Ind_k=1\}\cap\cE_k}-\Pra{\{\bW'_{k,2}=s\}\cap\{\Ind'_k=1\cap\cE_k\}}|\\
		&\le |\Pra{\{\bW_{2}=s\}\cap\{\Ind_k=1\}\cap\cE_k'}-\Pra{\{\bW'_{k,2}=s\}\cap\{\Ind'_k=1\}\cap\cE_k'}|\\
		&\le 
		\max\{\Pra{\{\bW_{2}=s\}\cap\{\Ind_k=1\}\cap\cE_k'},\Pra{\{\bW'_{k,2}=s\}\cap\{\Ind'_k=1\}\cap\cE_k'}\}\\
		&\le \Pra{\cE_k'}.
	\end{align*}
	In the same way we get
	\begin{align*}
		|\Pra{\Ind_k=1}-\Pra{\Ind'_k=1}|\le \Pra{\cE'_k}.
	\end{align*}
	Set
	\begin{align*}
		a&=\Pra{\{\bW_{2}=s\}\cap\{\Ind_k=1\}}\\
		x&=\Pra{\{\bW_{k,2}'=s\}\cap\{\Ind_k'=1\}}-\Pra{\{\bW_{2}=s\}\cap\{\Ind_k=1\}},\\	
		b&=\Pra{\Ind_k=1},\\
		y&=\Pra{\Ind_k'=1}-\Pra{\Ind_k=1}.
	\end{align*}
	Note that $0\le a\le 1$, $0\le b\le 1$, $|x|,|y|\le \Pra{\cE'_k}\le 8n^{-5}$. Moreover by Theorem~\ref{Thm:Qk:value}, since $6\Delta>\ell-\ln_qn$ and $n^{-5/(\ell-1)}=o(1-q)$, there exist positive constants $C_1,C_2,C_3$ such that $b=\Pra{\Ind_k=1}\ge C_1(1-q)^{\ell-1}$ and $b+y\ge  C_2(1-q)^{\ell-1}-8n^{-5}\ge C_3(1-q)^{\ell-1}$. Therefore
	\begin{align*}
		&\hspace{-1cm}|\Pra{\bW_{k,2}'=s}-\Pra{\bW_{k,2}=s}|=\\
		&=
		\left|
		\frac{\Pra{\{\bW'_{k,2}=s\}\cap\{\Ind'_k=1\}}}{\Pra{\Ind'_k=1}}
		-\frac{\Pra{\{\bW_{k,2}=s\}\cap\{\Ind_k=1\}}}{\Pra{\Ind_k=1}}
		\right|
		\\
		&=\left|\frac{a+x}{b+y}-\frac{a}{b}\right|\\
		&=\left|\frac{ab+bx-ab-ay}{b(b+y)}\right|\\
		&=\left|\frac{xb-ay}{b(b+y)}\right|\\
		&=\frac{|xb|+|ay|}{b(b+y)}\\
		&\le \frac{16 n^{-5}}{C_1C_3(1-q)^{2(\ell-1)}}\\
		&=O(n^{-5}(1-q)^{-2(\ell-1)})
	\end{align*}
	Therefore we may construct a coupling of ($\bW_{k,2}'$,$\bW_{k,2}$) in which we set 
	\[\Pra{\bW_{k,2}'=\bW_{k,2}=s}=\min\{\Pra{\bW_{k,2}'=s},\Pra{\bW_{k,2}=s}\}\]
	and other values we assign so that the marginal distributions fit those of $\bW_{k,2}'$ and $\bW_{k,2}$.
	In this coupling we have 
	\begin{align*}
		\Pra{|\bW_{k,2}'-\bW_{k,2}|\neq 0}&\le \sum_{s\in [|\Gamma_{2,k}|]}|\Pra{\bW_{k,2}'=s}-\Pra{\bW_{k,2}=s}|\\&\le nO(n^{-5}(1-q)^{-2(\ell-1)})=O(n^{-4}(1-q)^{-2(\ell-1)}).\end{align*}
	Thus, since $|\bW_{k,2}'-\bW_{k,2}|\le n$ we have
	\[
	\bE|\bW_{k,2}'-\bW_{k,2}|\le n \Pra{|\bW_{k,2}'-\bW_{k,2}|\neq 0}=O(n^{-3}(1-q)^{-2(\ell-1)}).
	\]
\end{proof}

\begin{proof}[Proof of Theorem~\ref{Thm:Main:Poisson}] Let $1-q$ be as in \eqref{Eq:Poisson:q:condition}. 
	By Lemma~\ref{Lem:CoupligProcesses}, the triangle inequality, and Lemma~\ref{Lem:InclusionDTV}
	\begin{align}
		\nonumber
		\dtv&\left(
		\Arr_{\cS}(\varPi_n),\operatorname{Po}(\bE \bW)
		\right)
		=\dtv\left(
		\Arr_{\cS}(\barPi[n]),\operatorname{Po}(\bE \bW)
		\right)\\ \nonumber
		&\le
		\dtv\left(\Arr_{\cS}(\barPi|_{[6\Delta+1,n-6\Delta]}),\operatorname{Po}(\lambda)\right)+
		\dtv\left(\Arr_{\cS}(\barPi|_{[6\Delta+1,n-6\Delta]}),\Arr_{\cS}(\barPi[n])\right)\\
		&\le 5n^{-5}+2\bE(\Arr_{\cS}(\varPi_{12\Delta+\ell}))+\dtv\left(\Arr_{\cS}(\barPi|_{[6\Delta+1,n-6\Delta]}),\Arr_{\cS}(\barPi[n])\right)\label{Eq:Final}
	\end{align}
	By \eqref{Eq:Poisson:Dtv:Suma}, Lemma~\ref{Lem:dtv:Coupling}, and \eqref{Eq:Lambda}
	\begin{multline}
		\dtv\left(\Arr_{\cS}(\barPi|_{[6\Delta+1,n-6\Delta]}),\Arr_{\cS}(\barPi[n])\right)
		=O((1\wedge n^{-1}(1-q)^{-(\ell-1)})\\
		\cdot(n(1-q)^{2\ell-2}+n\ln n(1-q)^{2\ell-3}+n(1-q)^{\ell}+n^{-3}(1-q)^{-2(\ell-1)}))
	\end{multline}
	For $n(1-q)^{\ell-1}\le 1$ we have $1\wedge n^{-1}(1-q)^{-(\ell-1)}=1$. Therefore, for $n(1-q)^{\ell-1}\le 1$ and $\ell\ge 3$,
	\begin{align*}
		n(1-q)^{2\ell-2}&\le n\cdot n^{-2} = o(1),\\
		n\ln n(1-q)^{2\ell-3}&\le n\ln n \cdot n^{-\frac{2\ell-3}{\ell-1}}=o(1)\\
		n(1-q)^{\ell}&\le n\cdot n^{-\ell/(\ell-1)} = o(1)
	\end{align*}
	In addition, by the first condition of \eqref{Eq:Poisson:q:condition}, we have
	\[n^{-3}(1-q)^{-2(\ell-1)}=o(1).\]
	Moreover, for  $n(1-q)^{\ell-1} > 1$ we have $1\wedge n^{-1}(1-q)^{-(\ell-1)}=n^{-1}(1-q)^{-(\ell-1)}$ and by the second condition of \eqref{Eq:Poisson:q:condition}
	\begin{align*}
		n^{-1}(1-q)^{-(\ell-1)}\cdot n(1-q)^{2\ell-2}&=(1-q)^{\ell-1}=o(1),\\
		n^{-1}(1-q)^{-(\ell-1)}\cdot n\ln n(1-q)^{2\ell-3}&= \ln n \cdot (1-q)^{\ell-2}=o(1)\\
		n^{-1}(1-q)^{-(\ell-1)}\cdot n(1-q)^{\ell}&=1-q= o(1).
	\end{align*}
	Last but not least, for  $n(1-q)^{\ell-1} > 1$
	\[n^{-1}(1-q)^{-(\ell-1)}\cdot n^{-3}(1-q)^{-2(\ell-1)}<n^{-1}=o(1).\]
	This proves that \eqref{Eq:Poisson:q:condition} implies
	\[\dtv\left(\Arr_{\cS}(\barPi|_{[6\Delta+1,n-6\Delta]}),\Arr_{\cS}(\barPi[n])\right)=o(1).\]
	Moreover by Theorem~\ref{Thm:Pinsky} and \eqref{Eq:Def:Delta}, for $q$ as in \eqref{Eq:Poisson:q:condition} \[\bE\Arr_{\cS}(\varPi_{12\Delta+\ell})\le \bE\Arr_{\bS_\ell}(\varPi_{12\Delta+\ell})= O(\Delta(1-q)^{\ell-1}) = O\left(\ln n (1-q)^{\ell-2}\right)=o(1). \]
	The two above equations with \eqref{Eq:Final} imply Theorem~\ref{Thm:Main:Poisson}.

\end{proof}

\section*{Acknowledgements}

I would like to express my sincere gratitude to John Sylvester for introducing me to the subject of Mallows permutations, for some fruitful discussions, and for providing suggestions that have improved the quality of this paper.




 
\end{document}